%% file: ModifiedOPs.tex
\documentclass[onefignum,onetabnum,final]{siamart190516}

\input{ModifiedOPs_shared}

\ifpdf
\hypersetup{
  pdftitle={OPsbyMatrixFactorizations},
  pdfauthor={T. S. Gutleb, S. Olver, R. M. Slevinsky}
}
\fi

\def\dmu{{{\rm d}\mu}}
\def\dx{{{\rm d}x}}
\def\supp{\operatorname{supp}}
\def\bs#1{\boldsymbol{#1}}
\def\bbR{{\mathbb R}}

\newcommand\extrafootertext[1]{%
    \bgroup
    \renewcommand\thefootnote{\fnsymbol{footnote}}%
    \renewcommand\thempfootnote{\fnsymbol{mpfootnote}}%
    \footnotetext[0]{#1}%
    \egroup
}

\begin{document}

\maketitle

\begin{abstract}
We describe fast algorithms for approximating the connection coefficients between a family of orthogonal polynomials and another family with a polynomially or rationally modified measure. The connection coefficients are computed via infinite-dimensional banded matrix factorizations and may be used to compute the modified Jacobi matrices all in linear complexity with respect to the truncation degree. A family of orthogonal polynomials with modified classical weights is constructed that support banded differentiation matrices, enabling sparse spectral methods with modified classical orthogonal polynomials. We present several applications and numerical experiments using an open source implementation which make direct use of these results.
\end{abstract}



\section{Introduction}
\extrafootertext{Communicated by Nira Dyn.}
Let $\mu$ be a positive Borel measure on the real line whose support contains an infinite number of points and has finite moments:
\begin{equation}\label{eq:finitemoments}
-\infty <\int_\bbR x^n\dmu(x) < \infty,\quad\forall n\in\mathbb{N}_0.
\end{equation}
We consider ${\bf P}(x) = (p_0(x) \quad p_1(x) \quad p_2(x) \quad \cdots)$ to be an orthonormal polynomial basis with respect to the inner-product:
\[
\langle f, g\rangle_\mu = \int_\bbR f(x) g(x)\dmu(x).
\]
That is, $\langle p_m, p_n\rangle_\mu = \delta_{m,n}$. This inner-product induces a norm, $\|f\|_\mu^2 = \langle f, f\rangle_\mu$, and the Hilbert space $L^2(\bbR, \dmu)$. We will assume that polynomials are dense in our Hilbert spaces, making them separable --- necessary and sufficient conditions for this are discussed in \cite[Theorem 2.3.3]{Akhiezer-65}.

Orthonormal polynomials are important in numerical analysis for their ability to represent $f\in L^2(\bbR, \dmu)$ via the  expansion:
\[
f(x) = \sum_{k=0}^\infty \langle f, p_k\rangle_\mu\,p_k(x) = {\bf P}(x) \bs{f},
\]
where $\bs{f} = (\langle f, p_0\rangle_\mu, \langle f, p_1\rangle_\mu, \ldots)^\top$, establishing the well-known isometric isomorphism between $L^2(\bbR, \dmu)$ and $\ell^2$ by identifying for each such function the unique corresponding infinite vector.  For simplicity we assume our function spaces are over the reals (that is, we only consider real-valued functions and vectors).

Given a nontrivial rational function $$r(x) = \frac{u(x)}{v(x)}$$ with polynomials $u$ and $v$, we will also consider ${\bf Q}(x) = (q_0(x) \quad q_1(x) \quad q_2(x) \quad \cdots)$ to be an orthonormal polynomial basis in $L^2(\bbR, r\dmu)$. The conditions on $\dmu$ are also imposed on $r \dmu$; in particular, $r(x)$ is nonnegative and nonsingular on $\supp(\mu)$. We wish to describe efficient algorithms for the connection problem between the original and the modified orthogonal polynomials, defined by the upper-triangular matrix $R$ in:
\[
{\bf P}(x) = {\bf Q}(x)R.
\]
That is, the computation of the coefficients $R_{k,n}$ (the entries of the matrix $R$) such that:
$$
p_n(x) = \sum_{k=0}^n q_k(x)R_{k,n}.
$$

Before doing so, it is important to relate this connection problem with the central object of a family of orthonormal polynomials, the Jacobi matrix. Every family of orthonormal polynomials ${\bf P}(x)$ has a Jacobi matrix $X_P$ (an infinite irreducible symmetric tridiagonal matrix), that implements multiplication by $x$:
\[
x{\bf P}(x) = {\bf P}(x)X_P.
\]
That is, if we denote the three-term recurrence as:
\[
x p_n(x) = \beta_{n-1} p_{n-1}(x) + \alpha_n p_n(x) + \beta_n p_{n+1}(x),\qquad p_{-1}(x) \equiv 0,
\]
then the nonzero entries of the Jacobi matrix $X_P$ are $(X_P)_{k,k} = \alpha_{k-1}$, $(X_P)_{k+1,k} = (X_P)_{k,k+1} = \beta_{k-1}$.
As $u$ and $v$ are polynomial it is now straightforward to define $U$ and $V$ by
\begin{align*}
u(x){\bf P}(x) &= {\bf P}(x)\underbrace{u(X_P)}_U,\qquad v(x){\bf P}(x) = {\bf P}(x)\underbrace{v(X_P)}_V
\end{align*}
as symmetric banded commuting multiplication matrices whose bandwidths are respectively equal to the degrees of the corresponding polynomials. In practice the entries of $U$ and $V$ can be efficiently computed via Clenshaw's algorithm~\cite{Clenshaw-9-118-55,Smith-19-33-65}.
We  also know that $x{\bf Q}(x) = {\bf Q}(x)X_Q$, though at this point we assume only $X_P$ is in hand.

\begin{definition}
For $0<p<\infty$, let $\ell^p$ denote the space of all sequences $\bs{v}\in\bbR^\infty$ satisfying:
\[
\sum_{n=0}^\infty |v_n|^p < \infty.
\]
By continuity, we shall consider $\ell^0$ to be the space of sequences with finitely many non-zeros.
\end{definition}

In order to establish the results of this paper we need to think of the infinite-dimensional matrices ($X_P$, $X_Q$, $U$, $V$) as operators acting on function spaces, in-particular $\ell^0$ and $\ell^2$, as we require access to functional calculus (square-roots, inverses, etc.). Note that an infinite-dimensional matrix is essentially an operator in which $\bs{e}_j^\top A \bs{e}_k \equiv \langle \bs{e}_j, A \bs{e}_k \rangle_{\ell^2}$ exists for all $j$ and $k$.

\begin{definition}
Let $P_n$ denote the canonical orthogonal projection:
\begin{equation}\label{eq:canonicalorthogonalprojection}
P_n = \begin{pmatrix} I_n & 0\\ 0 & 0\end{pmatrix},
\end{equation}
where $I_n$ is the $n\times n$ identity and the zero blocks are of conformable sizes.
\end{definition}

On $\ell^0$, triangular matrices with nonzero diagonals are trivially invertible. This can be seen by examining:
\[
P_nRP_n = \begin{pmatrix} R_n & 0\\ 0 & 0\end{pmatrix},
\]
and noting that $(R^{-1})_{j,k} = (R_n^{-1})_{j,k}$ provided that $j,k\le n$; that is, the principal finite section is taken to be sufficiently large.

\begin{theorem}[Gautschi~\cite{Gautschi-24-245-70}]\label{thm:connectioncoeffresult}
Let $X_P$ and $X_Q$ be the Jacobi matrices for the original and modified orthonormal polynomials, ${\bf P}(x)$ and ${\bf Q}(x)$, respectively. Then there exists a unique infinite invertible upper triangular operator $R : \ell^0 \rightarrow \ell^0$ such that:
\begin{enumerate}
\item $\mathbf{P}(x) = \mathbf{Q}(x) R$ and,\\
\item $RX_P = X_QR$.
\end{enumerate}
\end{theorem}
\begin{proof}
The proof is in two parts.
\begin{enumerate}
\item Since both ${\bf P}(x)$ and ${\bf Q}(x)$ are complete orthonormal bases of degree-graded polynomials, the degree-$n$ polynomial of one basis must result from the other as a unique linear combination of its first $n+1$ elements, resulting in a unique upper triangular conversion operator. $R^{-1}$ exists  because all upper triangular operators with nonzero diagonals are invertible on $\ell^0$.
\item This follows from:
\[
x{\bf P}(x) =
\left\{\begin{array}{c} {\bf P}(x)X_P = {\bf Q}(x) RX_P,\\
x{\bf Q}(x) R = {\bf Q}(x) X_QR.
\end{array}\right.
\]
\end{enumerate}
\end{proof}
This result shows that there is an infinite-dimensional similarity transformation between the two  Jacobi matrices (viewed as operators on $\ell^0$):
\[
X_Q = RX_PR^{-1},
\]
and the principal finite section of $X_Q$ may be computed in linear complexity. In particular, it depends only on $X_P$ and the main and first super-diagonals of $R$:
\begin{align*}
(X_Q)_{0,0} & = \frac{R_{0,0}(X_P)_{0,0} + R_{0,1}(X_P)_{1,0}}{R_{0,0}},\\
(X_Q)_{i,i+1} = (X_Q)_{i+1,i} & = \frac{R_{i+1,i+1}(X_P)_{i+1,i}}{R_{i,i}},\\
(X_Q)_{i,i} & = \frac{R_{i,i}(X_P)_{i,i} + R_{i,i+1}(X_P)_{i+1,i} - (X_Q)_{i,i-1}R_{i-1,i}}{R_{i,i}}.
\end{align*}
As a finite-dimensional similarity transformation preserves the spectrum, only the $(n-1)\times(n-1)$ principal finite section of $X_Q$ may be recovered by the $n\times n$ principal finite sections of $R$ and $X_P$ --- an observation of Kautsky and Golub~\cite{Kautsky-Golub-52-439-83} that follows from triangular and tridiagonal matrix multiplication:
\[
P_{n-1}X_QP_{n-1} = (P_{n-1}RP_n) (P_nX_PP_n) (P_n R^{-1}P_{n-1}).
\]
\begin{remark}\label{remark:quasimatrixqr}
In the modified Hilbert space $L^2(\mathbb{R},r\mathrm{d}\mu)$, ${\bf P}(x)$ is not orthonormal (apart from trivial measure modifications); hence, Theorem~\ref{thm:connectioncoeffresult} (1)~may be thought of as a $QR$ factorization in the sense of quasimatrices.
\end{remark}

The main contributions of this paper are outlined in Table~\ref{table:MOPs}, where we relate the computation of the connection coefficient matrix directly to different matrix factorizations in a unified manner.  We also consider approximating more general bounded modifications by rationals (Section~\ref{subsection:irrationalmeasuremodifications}), and when orthogonal polynomials have banded derivative relationships (Section~\ref{section:bandedderivatives}). Some of these factorizations are exactly realisable with finite-dimensional linear algebra (e.g. the Cholesky and $QR$ factorizations), while others are inherently infinite-dimensional (e.g. the reverse Cholesky and $QL$ factorizations), though in practice they can be approximated effectively using truncations (Section \ref{section:infinitealgorithms}). Finally, we present several applications and numerical experiments making use of the results described in this paper (Section \ref{section:applicationsnumexp}).

\begin{table}[htp]
\caption{A collection of connection problems resulting from banded matrix factorizations. }
\begin{center}
\begin{tabular}{c|cc}
$r(x) = u(x)/v(x)$ & Factorization(s) & Connections\\
\hline
\rule{0pt}{3ex}$v(x) = 1$ & $U = R^\top R$ & ${\bf P}(x) = {\bf Q}(x)R$\\
& & $u(x){\bf Q}(x) = {\bf P}(x) R^\top$\\[1ex]
\hline
\rule{0pt}{3ex}$v(x) = 1$, & $\sqrt{U} = QR$ & ${\bf P}(x) = {\bf Q}(x) R$\\
$\sqrt{u(x)}$ is a polynomial && $\sqrt{u(x)}{\bf Q}(x) = {\bf P}(x) Q$\\[1ex]
\hline
\rule{0pt}{3ex}$u(x) = 1$ & $V = L^\top L$ & ${\bf Q}(x) = {\bf P}(x)L^\top$\\
& & $v(x){\bf P}(x) = {\bf Q}(x)L$\\[1ex]
\hline
\rule{0pt}{3ex}$u(x) = 1$ & $\sqrt{V} = QL$ & ${\bf Q}(x) = {\bf P}(x)L^\top$\\
$\sqrt{v(x)}$ is a polynomial & & $\sqrt{v(x)}{\bf P}(x) = {\bf Q}(x)Q^\top$\\[1ex]
\hline
\rule{0pt}{3ex}Rational case 1 & $V = QL$ & ${\bf Q}(x) \RRone = {\bf P}(x)L^\top$\\
& $Q^\top U L^\top = \RRone^\top \RRone$ & $v(x){\bf P}(x) = {\bf Q}(x) \RRone Q^\top$\\[1ex]
\hline
\rule{0pt}{3ex}Rational case 2 & $V = L^\top L$ & ${\bf Q}(x) \RRtwo = {\bf P}(x)L^\top$\\
& $L U L^{-1} = \RRtwo^\top \RRtwo$ & $v(x){\bf P}(x) = {\bf Q}(x) \RRtwo L$\\[1ex]
\hline
\end{tabular}
\end{center}
\label{table:MOPs}
\end{table}

\subsection{Previous work}

Polynomial and rational measure modifications and its connection problem in Theorem~\ref{thm:connectioncoeffresult} are classical problems discussed at length in Gautschi's book on computational orthogonal polynomials \cite{Gautschi-04}. We mention some selected important advances on this problem without any claim to historical completeness. Much of the early development of methods to compute nonclassical orthogonal polynomials was motivated by their applications in quadrature methods.

As early as 1968, it was already known to Mysovskikh~\cite{Mysovskikh-178-1252-68} that the connection coefficients are the upper-triangular Cholesky factor of the Gram matrix, the matrix that corresponds to the left-hand side of our Eq.~\eqref{eq:rationalmodification} in the case of a rational modification. This abstract notion does not require polynomial or rational structure to the modification; and conversely, the sparsity structures present in the more specific modifications are absent in that discussion. The next year, Uvarov~\cite{Uvarov-9-1253-69} discusses the rational measure modification problem, uncovering the banded sparsity in the polynomial connection problem and in a Fourier--Christoffel version of the rational connection problem, $u(x){\bf Q}(x) = {\bf P}(x)C$, where $C$ is a banded matrix with lower bandwidth $\deg(u)$ and upper bandwidth $\deg(v)$. The banded sparsity of $C$ in fact characterizes the generalized Christoffel--Darboux formula, which explains why the modified basis must be multiplied by $u(x)$. While this connection problem has been solved since 1969, there are a few outstanding numerical issues. Firstly, while banded sparsity is important numerically, a general banded matrix such as $C$ must be further factored in order to allow the solution of linear systems. Secondly, Uvarov's method requires the computation of the roots and poles of $r(x)$, which are unstable with respect to perturbations. The method is also encumbered by the use of second-kind functions, whose numerical evaluation is not universally available. As a matter of computational complexity, the cost of computing the first $n$ columns of $C$ is $\mathcal{O}([\deg(u)+\deg(v)]^3n)$, as each column requires the solution of linear systems on the order of the bandwidth. We will note that Uvarov's Fourier--Christoffel connection problem may be reconstructed from either of lines (5) or (6) in Table~\ref{table:MOPs} as, for example, $C = UL^\top \RRone^{-1}$. While our table of factorizations may appear more complicated, the factors are all orthogonal or triangular, hence directly applicable and invertible.

By 1970, the work of Mysovskikh and Uvarov had begun disseminating in the West, though incompletely, with Gautschi~\cite{Gautschi-24-245-70} using the Cholesky factor of the Gram matrix to compute the modified Jacobi matrices for the purposes of modified Gaussian quadrature. Kumar~\cite{Kumar-28-769-74} independently proves Uvarov's Lemma~1, that reciprocal linear polynomial measure modifications connect the modified polynomials to the original polynomials by coefficients that are computed by second-kind functions. He uses this to compute modified first, second, and third kind Chebyshev quadrature rules, as the second-kind Chebyshev functions are particularly easy to compute. Price~\cite{Price-16-999-79} generalizes the work of Kumar to a general rational measure modification. In the case of a reciprocal polynomial measure modification, his Theorems 1 \& 2 uncover the banded sparsity in this special case of the connection problem; see line (3) in Table~\ref{table:MOPs}. This method corresponds to an upper-lower factorization of $v(X_P)$, though it relies on second-kind functions to compute the first row/column of the factors to allow such a factorization to proceed directly. Since Price's method relies on monomial moments, it is not suitable for large degree problems. In 1992, Skrzipek~\cite{Skrzipek-41-331-92} builds on Price's method by replacing the computation of monomial moments and second-kind functions with Gaussian quadrature for the starting columns in the upper-lower factorization of $v(X_P)$. This improves the stability of the method, but results in a computational complexity that depends on the degree of the quadrature rule. Inevitably, this may become quite large when the poles of the rational modification come close to the support of $\mu$.

It is well-known that Cholesky factorization is numerically stable, though the error may be large for an ill-conditioned matrix. In the context of polynomial modifications, if the roots of $u(x)$ are near  the support of $\mu$, the condition number of $U$ will be large. Kautsky and Golub~\cite{Kautsky-Golub-52-439-83} observed this phenomenon in the computation of $X_Q$ by the similarity transformation implied by Theorem~\ref{thm:connectioncoeffresult}, and provided two alternative algorithms that maintain a much higher accuracy for the computation of the modified Jacobi matrix. The first requires an orthogonal diagonalization of a principal finite section of $X_P$, while the second requires the polynomial $u(x)$ to be factored in terms of its roots, which allows Cholesky and $QR$ factorizations of successive modifications to take place. In related work, Buhmann and Iserles~\cite{Buhmann-Iserles-43-117-92} use the unshifted infinite-dimensional $QR$ algorithm to compute the $r(x) = x^2$ measure modification and draw interesting connections with certain Sobolev orthogonal polynomials. Elhay and Kautsky~\cite{Elhay-Kautsky-6-205-94} develop an inverse Cholesky method for a reciprocal degree polynomial modification of degree at most $2$, where a finite-rank correction based on second-kind function evaluation is added to the top of $V$ to allow the upper-lower factorization to proceed from the top down, that is, directly. Combined with a partial fraction decomposition of $r(x)$ and the summation technique of Elhay, Golub and Kautsky~\cite{Elhay-Golub-Kautsky-32-143-92}, the method enables a general rational weight modification. Elhay and Kautsky compare their method to Gautschi's method~\cite{Gautschi-36-547-81} of minimal solutions of certain recurrence relations and find the performance of methods complimentary: where one is poor the other is excellent, and vice versa.

Finally, we also mention that aside from the Gram--Schmidt type approaches which constitute the primary methods referenced in the literature and the secondary line of research projects on factorization methods outlined above, there have also been tertiary approaches to computing non-classical orthogonal polynomials independent of these two. An example of such a method is to solve the corresponding Riemann--Hilbert problem for the modified orthogonal polynomials as discussed in~\cite{Trogdon-Olver-36-174-16}.

\subsection{Infinite-dimensional matrix factorizations}

In finite dimensions, matrix factorizations are a powerful collection of tools for the solution of various ubiquitous types of problems. Some commonly used factorizations include the $LU$, $QR$, Cholesky, and Schur factorizations among many others~\cite{Golub-Van-Loan-13,Trefethen-Bau-97}. Depending on the specific needs of a particular application, there also may exist different algorithms for these factorizations favouring either speed or accuracy and allowing different levels of parallelization, a typical example being the contrast of Gram--Schmidt and Householder-based $QR$ factorization.

Cholesky factorization of a symmetric positive-definite square matrix $M = M^\top\in\mathbb{R}^{n\times n}$ represents $M$ as the unique product $R^\top R$, with an upper-triangular matrix $R$ with positive diagonal entries. Reverse Cholesky factorization produces an upper-lower factorization of the form $M = L^\top L$. $QR$ is one of the most well-studied and widely used matrix factorizations, in particular in the form of so-called thin or reduced $QR$ factorization, which allows the representation of a matrix $M\in\mathbb{R}^{m\times n}$ in terms of a product of a matrix $Q\in\mathbb{R}^{m\times n}$ with orthonormal columns and an upper-triangular matrix $R\in\mathbb{R}^{n\times n}$. 
This factorization is unique up to a {\em phase}, or a diagonal matrix $D\in\mathbb{R}^{n\times n}$ with entries $\pm 1$, since $QR = QDDR$ and the matrices $QD$ and $DR$ are respectively also orthogonal and upper-triangular.
The $QL$ factorization is defined analogously, where the matrix $L$ now is lower-triangular. For orthogonal-triangular matrix factorizations, we shall impose uniqueness by adopting the {\em positive phase} convention in $L$ and $R$, whereby $L_{i,i} > 0$ and $R_{i,i}>0$, $\forall i$.

In this work, we shall find it convenient to work with the infinite-dimensional analogues of Cholesky, reverse Cholesky, $QR$ and $QL$ factorizations. In a computing environment, one can represent infinite-dimensional linear operators as cached adaptive arrays~\cite{Olver-Townsend-57-14}, and this is facilitated and simplified in the case of banded operators. There is a distinction between Cholesky and $QR$ on the one hand and reverse Cholesky and $QL$ on the other, where finite sections of the former are {\em computable} (that is, with a finite number of operations), while finite sections of the latter are {\em not}. By perturbative arguments, we shall show in \S \ref{section:infinitealgorithms} when we expect (larger) finite section methods to approximate the true finite sections of the non-computable matrix factorizations.

If $A:\ell^2\to\ell^2$ is a bounded invertible linear operator, then $A = QR$, where $Q$ is not only an isometry\footnote{We use $\top$ to denote the adjoint operator, to emphasise that our function spaces are over the reals.}, $Q^\top Q = I$, but also orthogonal, $QQ^\top = I$, and $R$ is upper-triangular, bounded, and invertible. We refer the reader to~\cite{Deift-Li-Tomei-64-358-85,Hansen-254-2092-08,Colbrook-Hansen-143-17-19} for further details. If $A$ is banded and unbounded, then a $QR$ factorization still exists with an isometric $Q$ because it is independent of a diagonal column scaling, say $D$, such that $A D:\ell^2\to\ell^2$. If $A$ has dense column space, then $Q$ is also orthogonal as $Q$ and $A$ must have the same range, a fact that was used in~\cite{Aurentz-Slevinsky-410-109383-20}. Moreover, the diagonal column scaling shows that the upper-triangular operator $R$ is unbounded with the same growth as $A$.

Cholesky and reverse Cholesky factorizations may be extended analogously to the infinite-dimensional case, but there is a subtlety to the $QL$ factorization. For $QL$ to exist with orthogonal $Q$, it is sufficient\footnote{For non-invertible $A$ it is unclear under what conditions a $QL$ factorization exists or indeed is unique.} that $A$ be invertible~\cite{Webb-Thesis-17}.

\subsection{Contemporary Applications}

Almost all previous work mentions modified Gaussian quadrature as the archetypal application. But there are numerous reasons to be interested in modified classical orthogonal polynomials. Our interest can be divided into two classes of problems. Firstly, we are interested in the orthogonal structure in and on algebraic curves. Recent advances in this direction include a description of the orthogonal structures on quadratic~\cite{Olver-Xu-41-206-21}, planar cubic~\cite{Fasondini-Olver-Xu-21}, and a class of planar algebraic curves~\cite{Fasondini-Olver-Xu-151-369-23} and quadratic surfaces of revolution~\cite{Olver-Xu-89-2847-20}. Secondly, we are interested in the Koornwinder constructions of bivariate (and multivariate) orthogonal polynomials that are semi-classical. Orthogonal polynomials on half-disks and trapeziums are described in~\cite{Snowball-Olver-145-3-20} and on spherical caps in~\cite{Snowball-Olver-5-1-21} and we briefly introduce two more classes of polynomials on an annulus and a spherical band in \S\ref{subsection:annulus} and \S\ref{subsection:sphericalband}, respectively. Orthogonal polynomials with nonclassical weights can further be used for computing random matrix statistics for general invariant ensembles~\cite{Deift-99}, and in-particular, sampling their eigenvalues corresponds to a sequence of rational modifications of the measure~\cite{Olver-Nadakuditi-Trogdon-4-155002-15}.

\section{Rational measure modification via infinite-dimensional matrix factorizations}

In what follows, we denote the synthesis operator:
\[
{\cal S}:\ell^2\to L^2(\bbR,\dmu),~\bs{f}\mapsto {\bf P} \bs{f} = f,
\]
and the analysis operator:
\[
{\cal A}:L^2(\bbR, \dmu)\to\ell^2,~f\mapsto\int_\bbR  {\bf P}(x)^\top f(x) \dmu(x) = \bs{f}.
\]
For computational purposes, it is well-known that both operators have discrete analogues~\cite{Olver-Slevinsky-Townsend-29-573-20}.

Note that any bounded invertible operator has $QR$ and $QL$ factorizations and any symmetric positive definite operator has a Cholesky factorization.

\begin{proposition}[\cite{Deift-Li-Tomei-64-358-85,Colbrook-Hansen-143-17-19,Hansen-254-2092-08}]
If $A:\ell^2\to\ell^2$ is a bounded invertible  operator  then the $QR$ factorization  ($A = QR$ for $Q$ orthogonal and $R$ is invertible and bounded upper triangular) exists.  If $A$ is not invertible then $Q$ is only guaranteed to be an isometry and $R$ may not be invertible. 
\end{proposition}

\begin{proposition}[\cite{Webb-Thesis-17}\footnote{There is an issue in Webb's proof for the non-invertible case.}]\label{prop:QL}%
If $A:\ell^2\to\ell^2$ has a bounded inverse, then the $QL$ factorization ($A = QL$ for $Q$ orthogonal and $L$ lower triangular) exists.
\end{proposition}

In the infinite-dimensional setting there are two possible definitions for positive definite (which are equivalent in finite-dimensions) as outlined in the following definition:

\begin{definition}
A linear operator $A:\ell^2\to\ell^2$ is:
\begin{enumerate}
\item {\em positive semi-definite} if $\langle\bs{v}, A \bs{v}\rangle \ge 0$, for all $\bs{v} \in \ell^2$;
\item {\em positive definite} if $\langle\bs{v}, A \bs{v}\rangle > 0$, for all nonzero $\bs{v} \in \ell^2$; and,
\item {\em strictly positive definite} if $\exists M>0$ such that $\langle\bs{v}, A \bs{v}\rangle \ge M\langle\bs{v}, \bs{v}\rangle$, for all $\bs{v} \in \ell^2$.
\end{enumerate}
\end{definition}

\begin{proposition}[\cite{Chui-Ward-Smith-5-1-82,Goodman-et-al-35-233-95,Goodman-et-al-18-331-98}]
Let $A:\ell^2\to\ell^2$ be a self-adjoint linear operator with dense domain in $\ell^2$. If $A = R^\top R$ with upper-triangular $R$, then:
\begin{enumerate}
\item $A$ is positive semi-definite;
\item $R$ is unique and invertible on $\ell^0$ only if $A$ is positive definite; and,
\item $R$ is unique and invertible on $\ell^2$ only if $A$ is strictly positive definite.
\end{enumerate}
\end{proposition}

\begin{corollary}\label{cor:LtL}%
Let $A:\ell^2\to\ell^2$ be a self-adjoint linear operator with dense domain in $\ell^2$. If $A$ is invertible and $A = L^\top L$ with lower-triangular $L$, then:
\begin{enumerate}
\item $A^{-1}$ is positive semi-definite;
\item $L$ is unique and invertible on $\ell^0$ only if $A^{-1}$ is positive definite; and,
\item $L$ is unique and invertible on $\ell^2$ only if $A^{-1}$ is strictly positive definite.
\end{enumerate}
\end{corollary}

As occurs in unbounded domains, Jacobi matrices may be unbounded linear operators. Thus, our analytical framework for functional calculus must include bounded and unbounded functions of bounded and unbounded operators. Thus we require the Borel functional calculus, which will allow us to define functions of the Jacobi matrix through the projection-valued measure, $\dmu_{\bf P}(x) = {\bf P}(x)^\top {\bf P}(x)\dmu(x)$. We refer the interested reader to Reed and Simon~\cite[\S VIII.3]{Reed-Simon-80} for further details.

\begin{definition}[p.~263 in \cite{Reed-Simon-80}]
Let $r$ be a measurable function with finite $\mu$-moments:
\[
-\infty<\int_\bbR x^n r(x)\dmu(x) < \infty,\quad\forall n\in\mathbb{N}_0.
\]
We define:
\[
r(X_P) := \int_\bbR r(x) \dmu_{\bf P}(x).
\]
\end{definition}

If $r$ is a polynomial, then $r(X_P)$ is also realized by operator composition and the Clenshaw algorithm, as alluded to earlier, and if $r(x) = \frac{u(x)}{v(x)}$ is a rational function, then it follows from $\mathbf{P}(x) r(X_P) = r(x) \mathbf{P}(x) = \frac{u(x)}{v(x)} \mathbf{P}(x)$ that:
\[
r(X_P) = \underbrace{u(X_P) v(X_P)^{-1}}_{UV^{-1}} = \underbrace{v(X_P)^{-1} u(X_P)}_{ V^{-1} U}.
\]
\begin{proposition}\label{prop:rposdef}
If $r\dmu$ is a positive Borel measure with finite moments, as in Eq.~\eqref{eq:finitemoments}, then $r(X_P)$ is symmetric positive definite.
\end{proposition}
\begin{proof}
Symmetry of $r(X_P)$ is a direct consequence of the orthonormality of $\mathbf{P}(x)$. To show positive definiteness we observe that for any nonzero $\bs{v} \in \ell^2$ we have a nonzero $v(x) = {\bf P}(x) \bs{v}\in L^2(\bbR, \dmu)$:
\[
\langle\bs{v}, r(X_P)\bs{v}\rangle = \int_\bbR v(x)^2 r(x) \dmu(x)>0.
\]
See also~\cite[Theorem VIII.5 (f)]{Reed-Simon-80}.
\end{proof}

From this last result we can translate the computation of the conversion operator $R$ between ${\bf P}(x)$ and ${\bf Q}(x)$ to computing a Cholesky factorization:

\begin{lemma}
If $r\dmu$ is a positive Borel measure with finite moments, then the Cholesky factorization of $r(X_P)$ encodes the conversion matrix $R$, i.e.:
\begin{equation}\label{eq:rationalmodification}
r(X_P) = R^\top R,
\end{equation}
if and only if ${\bf P}(x) = {\bf Q}(x) R$. Furthermore:
\[
r(x) {\bf Q}(x) = {\bf P}(x) R^\top.
\]
\end{lemma}

\begin{proof}
By Proposition~\ref{prop:rposdef}, $r(X_P)$ is a self-adjoint positive definite operator hence its Cholesky factorization exists and $R$ is invertible on $\ell^0$: this much was known to Mysovskikh~\cite{Mysovskikh-178-1252-68}, who called principal finite sections of $r(X_P)$ the Gram matrix. Thus ${\bf \tilde Q}(x) = {\bf P}(x) R^{-1}$ connects two families of graded polynomials with leading coefficients of the same sign. While $R^{-1}$ is not necessarily bounded on $\ell^2$ we have $R^{-1} \bs{e}_n \in \ell^2$ for all $n$. Thus we find for $\tilde{q}_n(x) = {\bf \tilde{Q}}(x) \bs{e}_n$:
\begin{align*}
\int_\bbR \tilde{q}_m(x)^\top r(x) \tilde{q}_n(x)\dmu(x) & = \bs{e}_m^\top R^{-\top} \int_\bbR {\bf P}(x)^\top r(x) {\bf P}(x)\dmu(x)  R^{-1} \bs{e}_n \\
 &= \bs{e}_m^\top R^{-\top}  r(X_P)  R^{-1} \bs{e}_n = \delta_{m,n},
\end{align*}
which shows by uniqueness ${\bf \tilde{Q}}(x) = {\bf Q}(x)$. For the other direction we have
\[
r(X_P) = \int_\bbR {\bf P}(x)^\top r(x) {\bf P}(x) \dmu(x) = R^\top \int_\bbR {\bf Q}(x)^\top  r(x) {\bf Q}(x) \dmu(x) R = R^\top R.
\]
The final statement follows since:
\[
r(x) {\bf Q}(x) = {\bf P}(x) r(X_P) R^{-1} = {\bf P}(x) R^\top.
\]
\end{proof}

The polynomial measure modification, where $r(x) = u(x)$ is a polynomial, corresponds to the case $V = I$, in which case the banded Cholesky factorization of $U$ returns the connection coefficients, showing line (1) of Table~\ref{table:MOPs}, which is equivalent to the result of Gautschi~\cite{Gautschi-24-245-70}. That is, in this setting one can establish the result using only finite-dimensional linear algebra.

In the case that $\sqrt{u(x)}$ is a polynomial we establish line (2) of Table~\ref{table:MOPs} as follows:

\begin{lemma}
Suppose $r(x) = u(x)$ where $\sqrt{u(x)}$ is a polynomial. The QR factorization of $\sqrt U$ encodes the conversion matrix $R$, i.e.:
\begin{equation}
\sqrt{U} = Q R,
\end{equation}
if and only if ${\bf P}(x) = {\bf Q}(x) R$. Furthermore:
\[
\sqrt{u(x)} {\bf Q}(x) = {\bf P}(x) Q.
\]
\end{lemma}
\begin{proof}
First note that $R_{n,n}>0$ because the Cholesky factor of $r(X_P)$, the square of a polynomial, is invertible on $\ell^0$ (exactly as argued above). The first result follows as above by observing that
\[
U = \sqrt{U}^\top \sqrt{U}= R^\top Q^\top Q R = R^\top R.
\]
The second result follows from:
\[
\sqrt{u(x)} {\bf Q}(x) = {\bf P}(x) \sqrt{U} R^{-1} =  {\bf P}(x) Q.
\]
\end{proof}

We now consider $r(x) = 1/v(x)$ to be a reciprocal polynomial. Note that the entries of $V^{-1}$ are not computable\footnote{They can be approximated by approximating $v(x)^{-1}$ by a polynomial but this degenerates if $v(x)$ has zeros near $\supp(\mu)$.}. However, lines (3) and (4) of Table~\ref{table:MOPs} can be deduced from factorizations of $V$ directly: we know from the fact that $r(x)$ has no poles in $\supp(\mu)$ that $V$ has a bounded inverse. By Proposition~\ref{prop:QL} and Corollary~\ref{cor:LtL}, there exist $QL$ and reverse Cholesky factorizations and then the proofs follow along the same lines as above.

For other rational modifications we use factorizations of the left-hand side of Eq.~\eqref{eq:rationalmodification} that respect the upper-triangular structure of the Cholesky factors on the right-hand side. Two avenues to proceed require the infinite-dimensional $QL$ and reverse Cholesky algorithms.

\begin{theorem}[Rational case 1]
Suppose $r \in L^\infty(\bbR, \dmu)$ and let  $V = QL$. Then
\[
Q^\top UL^\top = \RRone^\top \RRone,
\]
if and only if ${\bf Q}(x)\RRone = {\bf P}(x) L^\top$. Furthermore:
\[
v(x) {\bf P}(x) = {\bf Q}(x) \RRone Q^\top.
\]
\end{theorem}
\begin{proof}
As above, nonsingularity of $r(x)$ on $\supp(\mu)$ guarantees that $V$ is invertible, and we find $V^{-1} = L^{-1}Q^\top$. Define ${\bf \tilde{Q}}(x) := {\bf P}(x) L^\top \RRone^{-1}$ and $\tilde{q}_n(x)$ as above, so that
\begin{align*}
\int_\bbR {\tilde q}_m(x)^\top {u(x) \over v(x)} {\tilde q}_n(x) \dmu(x) &=    \bs{e}_m^ \top \RRone^{-\top} L V^{-1} U  L^{\top} \RRone^{-1}  \bs{e}_n = \delta_{m,n}.
\end{align*}
For the other direction we have $R = \RRone L^{-\top}$ hence
\[
Q^\top U L^\top  = L V^{-1} U  L^\top  = L R^\top \int_\bbR {\bf Q}(x)^\top r(x) {\bf Q}(x) \dmu(x) R L^\top  =\RRone^\top \RRone.
\]
The final statement follows since $V = V^\top = L^\top Q^\top$ and hence:
\[
v(x) {\bf P}(x) = {\bf P}(x)  V = {\bf Q}(x) \RRone L^{-\top}  V = {\bf Q}(x) \RRone Q^\top.
\]
\end{proof}

\begin{theorem}[Rational case 2] Suppose $r \in L^\infty(\bbR, \dmu)$ and let  $V = L^\top L$. Then
\[
L UL^{-1} = \RRtwo^\top \RRtwo,
\]
if and only if ${\bf Q}(x)\RRtwo= {\bf P}(x) L^\top$. Furthermore:
\[
v(x) {\bf P}(x) = {\bf Q}(x) \RRtwo L.
\]
\end{theorem}

\begin{proof}
Note for the last time that nonsingularity of $r(x)$ on $\supp(\mu)$ guarantees that $V$ is invertible, and we find $V^{-1} = L^{-1}L^{-\top}$. Define ${\bf \tilde{Q}}(x) := {\bf P}(x) L^\top \RRtwo^{-1}$ and $\tilde{q}_n(x)$ as above, so that
\begin{align*}
\int_\bbR \tilde{q}_m(x)^\top {u(x) \over v(x)} \tilde{q_n}(x) \dmu(x) &= \bs{e}_m^ \top \RRtwo^{-\top} L   V^{-1} U  L^{\top} \RRtwo^{-1}  \bs{e}_n = \delta_{m,n}.
\end{align*}
For the other direction we have $R = \RRtwo L^{-\top}$ hence
\[
L U L^{-1}  = L  U V^{-1}  L^\top  = L \int_\bbR {\bf P}(x)^\top r(x) {\bf P}(x) \dmu(x) L^\top  = \RRtwo^\top \RRtwo.
\]
The final statement follows since $V = V^\top = L^\top L$ and hence:
\[
v(x) {\bf P}(x) = {\bf P}(x)  V = {\bf Q}(x) \RRtwo L^{-\top}  V = {\bf Q}(x) \RRtwo L.
\]
\end{proof}

Through banded matrix factorizations, Table~\ref{table:MOPs} generalizes many familiar identities for the classical orthogonal polynomials. We give a few examples using the orthonormal generalized Laguerre polynomials~\cite[\S 18.3]{Olver-et-al-NIST-10}, related to the standard normalization by:
\[
\tilde{L}_n^{(\alpha)}(x) = \sqrt{\frac{\Gamma(n+1)}{\Gamma(n+\alpha+1)}}L_n^{(\alpha)}(x).
\]
Given $X_{\tilde{L}^{(\alpha)}}$, whose entries are determined by:
\[
x\tilde{L}_n^{(\alpha)}(x) = -\sqrt{n(n+\alpha)}\tilde{L}_{n-1}^{(\alpha)}(x) + (2n+\alpha+1)\tilde{L}_n^{(\alpha)}(x) - \sqrt{(n+1)(n+\alpha+1)}\tilde{L}_{n+1}^{(\alpha)}(x),
\]
four other operators arise by Cholesky and $QR$ factorization. These are:
\begin{enumerate}
\item The raising operator that is the upper Cholesky factor of $X_{\tilde{L}^{(\alpha)}} = R^\top R$:
\[
\tilde{L}_n^{(\alpha)}(x) = \sqrt{n+\alpha+1}\tilde{L}_n^{(\alpha+1)}(x) - \sqrt{n}\tilde{L}_{n-1}^{(\alpha+1)}(x),
\]
\item The lowering operator that is the transpose of the raising operator:
\[
x\tilde{L}_n^{(\alpha+1)}(x) = \sqrt{n+\alpha+1}\tilde{L}_n^{(\alpha)}(x) - \sqrt{n+1}\tilde{L}_{n+1}^{(\alpha)}(x),
\]
\item The iterated raising operator that is the triangular factor of $X_{\tilde{L}^{(\alpha)}} = QR$:
\begin{align*}
\tilde{L}_n^{(\alpha)}(x) & = \sqrt{(n+\alpha+1)(n+\alpha+2)}\tilde{L}_n^{(\alpha+2)}(x)\\
& - 2\sqrt{n(n+\alpha+1)}\tilde{L}_{n-1}^{(\alpha+2)}(x) + \sqrt{n(n-1)}L_{n-2}^{(\alpha+2)}(x),
\end{align*}
\item The corresponding orthogonal factor~\cite[Corollary 4.8]{Olver-Slevinsky-Townsend-29-573-20}:
\begin{align*}
x\tilde{L}_n^{(\alpha+2)}(x) & + \sqrt{\frac{n+1}{n+\alpha+2}}\tilde{L}_{n+1}^{(\alpha)}(x)\\
& = (\alpha+1)\sum_{\ell=0}^n\sqrt{\frac{\Gamma(\ell+\alpha+1)}{\Gamma(\ell+1)} \frac{\Gamma(n+1)}{\Gamma(n+\alpha+3)}} \tilde{L}_\ell^{(\alpha)}(x).
\end{align*}
\end{enumerate}

\begin{remark}
In the general rational cases, it is important to relate the main and first super-diagonals of $R$ to the factored forms of the connection coefficients to efficiently compute the modified Jacobi matrix $X_Q$. In the first general rational case, $R = \RRone L^{-\top}$, so:
\begin{align*}
R_{i,i} & = (\RRone)_{i,i}/L_{i,i},\\
R_{i,i+1} & = \frac{(\RRone)_{i,i+1} - R_{i,i}L_{i+1,i}}{L_{i+1,i+1}}.
\end{align*}
In the second general rational case, the same formulas hold with $\RRone \leftrightarrow \RRtwo$.

Finally, we observe that $\RRone$ is an upper-triangular {\em banded} matrix with upper bandwidth at most $\deg(u)+\deg(v)$. This inference follows from the same bound on the lower bandwidth of $Q^\top U L^\top$. Similarly, $\RRtwo$ is an upper-triangular {\em banded} matrix with upper bandwidth at most $\deg(u)$, a result that follows from comparison with $LUL^{-1}$. Since $\RRone$ and $\RRtwo$ are banded matrices, their principal finite sections are computable in $\mathcal{O}(n)$ flops.
\end{remark}

\subsection{$L^\infty$ measure modifications}\label{subsection:irrationalmeasuremodifications}

It is reasonable to consider irrational measure modifications $r(x)$. We wish to characterize the rational approximation error in $L^\infty(\bbR,\dmu)$ to the $2$-norm error in approximating the infinite-dimensional matrix $r(X_P)$. In this way, we will understand when polynomial and rational approximants construct nearby matrices of measure modifications to the true problems. This confers a significant computational advantage as then we are able to leverage the banded sparsity in the nearby problem.

\begin{lemma}[Proposition 1, \S VIII.3 in \cite{Reed-Simon-80}]
For any $r\in L^\infty(\bbR, \dmu)$:
\[
\|r(X_P)\|_2 = \|r\|_{L^\infty(\bbR, \dmu)}.
\]
\end{lemma}
A consequence of this is that we can control the error in approximating a measure modification by a rational function:
\begin{theorem}
Given $\epsilon>0$ and a measure modification $r\in L^\infty(\bbR, \dmu)$, if it is approximated by a rational function $u/v$ such that:
\[
\left\|r - \frac{u}{v}\right\|_{L^\infty(\bbR, \dmu)} \le \epsilon \|r\|_{L^\infty(\bbR, \dmu)},
\]
then:
\[
\left\|r(X_P) - V^{-1}U\right\|_2 \le \epsilon \|r(X_P)\|_2.
\]
\end{theorem}

As an example, we consider a Jacobi polynomial approximation to an analytic function on $[-1,1]$. By classical results of approximation theory~\cite[Theorems 8.2 \& 12.1]{Trefethen-12}, the canonical finite-dimensional interpolation and projection errors are dominated by exponential convergence; in particular, the degree dependence is $\mathcal{O}(\log\epsilon^{-1})$ as $\epsilon\to0$. Thus, to relative error $\epsilon$, the matrix $r(X_P)$ is well-approximated by $U$ (and thereby trivially $V^{-1}U$) of effectively finite bandwidth.

On unbounded domains, it is impossible to approximate a bounded measure modification in the uniform norm by polynomials, partially motivating our interest in rational approximation.

\subsection{Measure modifications with $M$-matrix connection coefficients}

We include in this section a result on a property of the measure modification that implies a property of the connection coefficients. Since this property is inherently entry-wise, it is understood that in this section we work with principal finite sections of the connection coefficients; hence, all matrices are finite-dimensional.

\begin{definition}
A positive-definite matrix $A$ is an $M$-matrix if $A_{i,i} > 0$, $A_{i,j}\le0$ for $i\ne j$. A symmetric $M$-matrix is also known as a Stieltjes matrix. 
\end{definition}

Triangular $M$-matrices and their inverses have tighter-than-general $2$-norm condition number estimates~\cite{Qi-56-105-84}; and it has been useful to identify this property in the setting of a connection problem~\cite{Klippenstein-Slevinsky-403-113831-22}.

\begin{lemma}[Fiedler and Pt\'ak~\cite{Fiedler-Ptak-12-382-62}]\label{lemma:MmatrixLU}
Let $A$ be an $M$-matrix. Then $A=LU$ and both $L$ and $U$ are $M$-matrices.
\end{lemma}
In particular, this means that the Cholesky factors of a Stieltjes matrix are $M$-matrices.

\begin{definition}[Fiedler and Schneider~\cite{Fiedler-Schneider-13-185-83}]
A smooth function $f :(0,\infty)\to[0,\infty)$ is {\em completely monotonic} if $(-1)^kf^{(k)}(x)\ge0$ for $x>0$ and $k\in\mathbb{N}_0$.
\end{definition}

\begin{theorem}[Fiedler and Schneider~\cite{Fiedler-Schneider-13-185-83}]\label{theorem:completelymonotonic}
Let $f$ be positive on $(0,\infty)$ and $f'$ be completely monotonic. Then $A$ is an $M$-matrix if and only if $f(A)$ is an $M$-matrix.
\end{theorem}

\begin{corollary}
Let $r(x) = f(s(x))$ where $f$ is positive on $(0,\infty)$ and $f'$ is completely monotonic and $s$ is smooth. If principal finite sections of $s(X_P)$ are Stieltjes matrices, then principal finite sections of $r(X_P)$ are Stieltjes matrices if and only if principal finite sections of $R$ are $M$-matrices.
\end{corollary}
\begin{proof}
This corollary constructively combines Lemma~\ref{lemma:MmatrixLU} and Theorem~\ref{theorem:completelymonotonic}.
\end{proof}

Given the Jacobi polynomials, it is easy to show that principal finite sections of $I-X_P$ and $I-X_P^2$ are Stieltjes matrices, corresponding to $s(x) = 1-x$ and $s(x) = 1-x^2$, respectively. For the Laguerre polynomials, principal finite sections of $X_P$ itself are Stieltjes matrices. A few concrete examples of $f$ positive and $f'$ completely monotonic include:
\begin{enumerate}
\item $f(x) = x^\lambda$ for $0\le\lambda\le1$; and,
\item $\displaystyle f(x) = \frac{x}{x+\gamma}$ for $\gamma>0$.
\end{enumerate}

\subsection{Banded derivatives of modified classical orthogonal polynomials}\label{section:bandedderivatives}

Classical orthogonal polynomials are characterized by Bochner~\cite{Bochner-29-730-29} and Krall~\cite{Krall-4-705-38} as the polynomial solutions of the degree-preserving second-order linear differential equation:
\begin{equation}\label{eq:COPEIG}
\left(\sigma \DD^2 + \tau \DD + \lambda_n\right) p_n = 0,
\end{equation}
where $\sigma$ and $\tau$ are polynomials in $x$ independent of $n$ that satisfy $\deg(\sigma) \le 2$ and $\deg(\tau) \le 1$, and the eigenvalues $\lambda_n = -\frac{n}{2}\left[(n-1)\sigma''+2\tau'\right]\ge0$. We call this factorization degree-preserving because the degrees of the polynomial variable coefficients do not exceed the orders of the respective differential operators.

In addition, the measure is expressed in terms of a positive weight function, $\dmu(x) = w(x){\rm\,d}x$ supported on the (possibly unbounded) interval $(a,b)$, that satisfies the first-order Pearson differential equation:
\begin{equation}\label{eq:Pearson}
\DD(\sigma w) = \tau w.
\end{equation}
A useful property of $\sigma$ is that it is zero at finite boundary points: for $\chi$ equal to $a$ or $b$ if they are finite we have
$
\sigma(\chi) w(\chi) = 0,
$
annihilating boundary terms in integration by parts.

By considering the self-adjoint form of Eq.~\eqref{eq:COPEIG}:
\begin{equation}\label{eq:COPSAEIG}
(-\DD)\left(\sigma w\DD\right)p_n = \lambda_n w p_n,
\end{equation}
it follows that the polynomials $\DD p_n$ are orthogonal with respect to $L^2(\bbR, \sigma\dmu)$. Thus the characterization of classical orthogonal polynomials is often stated as those polynomials whose derivatives are also orthogonal polynomials.

It follows that differentiation of the classical orthogonal polynomials can be represented by infinite-dimensional banded matrices:
\begin{equation}\label{eq:classicalderivative}
\DD {\bf P}(x) = {\bf P}'(x) D_P^{P'},
\end{equation}
where $\DD$ denotes the derivative operator, $ {\bf P}(x)$ are the orthonormal polynomials with respect to $\dmu(x) = w(x){\rm\,d}x$,  $ {\bf P}'(x)$ are the orthonormal polynomials with respect to $\sigma \dmu$ and $D_P^{P'}$ is a banded matrix with nonzero entries only on the first super-diagonal. Moreover, the super-diagonal entries of $D_P^{P'}$ may be found by Eq.~\eqref{eq:COPSAEIG}:
\[
(D_P^{P'})^\top D_P^{P'} = \Lambda,\qquad \Lambda = \diag(\lambda_0,\lambda_1,\ldots).
\]
From a computational perspective, banded multiplication, raising, and differentiation enable sparse spectral methods to be constructed for linear differential equations with (polynomial) variable coefficients~\cite{Olver-Townsend-55-462-13,Olver-Slevinsky-Townsend-29-573-20}.

In light of the current work on rationally modified orthogonal polynomials, we may ask:
\begin{quotation}
Is there a basis in which the derivative of the rationally modified polynomials is banded?
\end{quotation}
The following theorem shows that banded derivatives exist for this problem and also more generally in the case of classical weights modified by algebraic powers of a product of polynomials. Let $\bs{\alpha} = (\alpha_1,\ldots,\alpha_d)$ be a multi-index and $\bs{u}(x) = (u_1(x),\ldots,u_d(x))$ be a collection of polynomials. We will use the following notation:
\[
\bs{u}^{\bs{\alpha}}(x) = \prod_{i=1}^d u_i^{\alpha_i}(x),
\]
and the shorthand $\displaystyle \deg(\bs{u}^{\bs{\alpha}}) = \sum_{i=1}^d \deg(u_i^{\alpha_i})$ if $\bs{\alpha}\in\mathbb{N}_0^d$. Define $\lambda:\mathbb{R}^d\to\mathbb{N}_0^d$ by:
\[
[\lambda(\bs{\alpha})]_i = \left\{\begin{array}{cc} 1 & \alpha_i\ne0,\\ 0 & \alpha_i=0.\end{array}\right.
\]
\begin{theorem}\label{theorem:bandedderivatives}
Let $w(x)$ be a classical orthogonal polynomial weight, let $\sigma(x)$ be the corresponding degree at most $2$ polynomial, let:
\[
\dmu^{(\bs{\alpha},\beta)}(x) = w^{(\bs{\alpha},\beta)}(x) {\rm\,d}x= \bs{u}^{\bs{\alpha}}(x)\sigma^\beta(x)w(x) {\rm\,d}x,
\]
with $\bs{u}^{\bs{\alpha}}(x)\sigma^\beta(x)$ such that $\dmu^{(\bs{\alpha},\beta)}(x)$ is a positive Borel measure with finite moments. Further, let ${\bf P}^{(\bs{\alpha},\beta)}(x)$ be orthonormal polynomials with respect to $L^2(\bbR, \dmu^{(\bs{\alpha},\beta)})$ for any real $\bs{\alpha}$ and $\beta$ such that, in particular:
\[
\lim_{x\to a^+} x^n w^{(\bs{\alpha}+\lambda(\bs{\alpha}),\beta+1)}(x) = \lim_{x\to b^-} x^n w^{(\bs{\alpha}+\lambda(\bs{\alpha}),\beta+1)}(x) = 0,\quad\forall n\in\mathbb{N}_0.
\]
Then:
\begin{equation}
\DD {\bf P}^{(\bs{\alpha},\beta)}(x) = {\bf P}^{(\bs{\alpha}+\lambda(\bs{\alpha}), \beta+1)}(x)D_{(\bs{\alpha}, \beta)}^{(\bs{\alpha}+\lambda(\bs{\alpha}), \beta+1)},
\end{equation}
where $D_{(\bs{\alpha}, \beta)}^{(\bs{\alpha}+\lambda(\bs{\alpha}), \beta+1)}$ is strictly upper triangular and has upper bandwidth at most $\deg(\bs{u}^{\lambda(\bs{\alpha})})+1$.
\end{theorem}
\begin{proof}
Differentiation of polynomials reduces the degree; hence, $D_{(\bs{\alpha}, \beta)}^{(\bs{\alpha}+\lambda(\bs{\alpha}), \beta+1)}$ is strictly upper triangular. Consider at first the relaxed problem:
\[
\DD {\bf P}^{(\bs{\alpha},\beta)}(x) = {\bf P}^{(\bs{\alpha}+\bs{1},\beta+1)}(x)D_{(\bs{\alpha},\beta)}^{(\bs{\alpha}+\bs{1},\beta+1)}.
\]
To find its upper bandwidth, we note that: 
\begin{align*}
& (D_{(\bs{\alpha},\beta)}^{(\bs{\alpha}+\bs{1},\beta+1)})_{m,n}\\%
& = \int_a^b p_m^{(\bs{\alpha}+\bs{1},\beta+1)}(x) \DD p_n^{(\bs{\alpha},\beta)}(x) w^{(\bs{\alpha}+\bs{1},\beta+1)}(x){\rm\,d}x,\\%
& = -\int_a^b \DD\left[p_m^{(\bs{\alpha}+\bs{1},\beta+1)}(x)w^{(\bs{\alpha}+\bs{1},\beta+1)}(x)\right] p_n^{(\bs{\alpha},\beta)}(x) {\rm\,d}x,\\%
& = -\int_a^b \left\{\DD\left[p_m^{(\bs{\alpha}+\bs{1},\beta+1)}(x)\right]w^{(\bs{\alpha}+\bs{1},\beta+1)}(x)\right.\\%
&\quad\qquad + p_m^{(\bs{\alpha}+\bs{1},\beta+1)}(x)\DD\left[\bs{u}^{\bs{\alpha}+\bs{1}}(x)\right]\sigma^{\beta+1}(x)w(x)\\%
&\quad\qquad \left. + p_m^{(\bs{\alpha}+\bs{1},\beta+1}(x)\bs{u}^{\bs{\alpha}+\bs{1}}(x)\DD\left[\sigma^{\beta+1}(x) w(x)\right] \right\} p_n^{(\bs{\alpha},\beta)}(x) {\rm\,d}x,\\%
& = -\int_a^b \left\{\DD\left[p_m^{(\bs{\alpha}+\bs{1},\beta+1)}(x)\right]\bs{u}^{\bs{1}}(x)\sigma(x)\right.\\
&\quad\qquad + p_m^{(\bs{\alpha}+\bs{1},\beta+1)}(x)\sum_{i=1}^d(\alpha_i+1)u'_i(x)\prod_{\substack{j=1\\j\ne i}}^du_j(x)\sigma(x)\\
&\quad\qquad \left. + p_m^{(\bs{\alpha}+\bs{1},\beta+1)}(x)\bs{u}^{\bs{1}}(x)(\tau+\beta{\sigma'})\right\} p_n^{(\bs{\alpha},\beta)}(x) w^{(\bs{\alpha},\beta)}(x){\rm\,d}x.
\end{align*}
It follows that the first term is a degree at most $m+\deg(\bs{u}^{\bs{1}})+\max\{\deg(\sigma)-1,\deg(\tau)\}$ polynomial, proving that $(D_{(\bs{\alpha},\beta)}^{(\bs{\alpha}+\bs{1},\beta+1)})_{m,n}=0$ if $n>m+\deg(\bs{u}^{\bs{1}})+1$, by orthogonality of $p_n^{(\bs{\alpha},\beta)}(x)$ with all polynomials of lesser degree. If any of the exponents $\alpha_i$ is zero, then this overestimates the upper bandwidth, as $u_i^0(x) \equiv 1$, no matter how large its degree.
\end{proof}

While Theorem~\ref{theorem:bandedderivatives} shows that these classes of orthogonal polynomials have banded derivatives, it does not provide an algorithm for computations. In fact, there are many different formulas to compute these based on connection coefficients. For example, if:
\[
{\bf P}^{(\bs{\alpha},\beta)}(x) = {\bf P}^{(\bs{\gamma},\delta)}(x)R_{(\bs{\alpha},\beta)}^{(\bs{\gamma},\delta)},
\]
then it is true that:
\begin{equation}
D_{(\bs{\alpha},\beta)}^{(\bs{\alpha}+\lambda(\bs{\alpha}),\beta+1)} = R_{(\bs{\gamma}+\lambda(\bs{\gamma}),\delta+1)}^{(\bs{\alpha}+\lambda(\bs{\alpha}),\beta+1)} D_{(\bs{\gamma},\delta)}^{(\bs{\gamma}+\lambda(\bs{\gamma}),\delta+1)} R_{(\bs{\alpha},\beta)}^{(\bs{\gamma},\delta)}.
\end{equation}
If $\bs{\gamma} = \bs{0}$, then $D_{(\bs{\gamma},\delta)}^{(\bs{\gamma}+\lambda(\bs{\gamma}),\delta+1)}$ is a classical orthogonal polynomial differentiation matrix.

In the proof of Theorem~\ref{theorem:bandedderivatives}, we found that:
\[
\DD(\sigma^{\beta+1}w) = (\tau+\beta{\sigma'})\sigma^\beta w,
\]
which shows that $\sigma^\beta(x)w(x)$ is also a classical weight. We include $\sigma^\beta(x)$ to enable convenient notation for higher order derivatives.

Without loss of generality, we discuss the banded higher order derivatives of rationally modified orthonormal polynomials. To represent a rational classical weight modification, suffice it to take $d=2$ and identify $\bs{u}(x) = (u(x), v(x))$ and $\bs{\alpha} = (1,-1)$. Then:
\[
\DD {\bf P}^{(1,-1,0)}(x) = {\bf P}^{(2,0,1)}(x)D_{(1,-1,0)}^{(2,0,1)}.
\]
It follows that:
\[
D_{(1,-1,0)}^{(2,0,1)} = R_{(0,0,1)}^{(2,0,1)} D_{(0,0,0)}^{(0,0,1)} (R_{(0,0,0)}^{(1,-1,0)})^{-1}.
\]
This formula finds the banded derivative $D_{(1,-1,0)}^{(2,0,1)}$ in terms of a polynomial weight modification $R_{(0,0,1)}^{(2,0,1)}$, the classical derivative $D_{(0,0,0)}^{(0,0,1)}$, and the inverse of the rational weight modification $R_{(0,0,0)}^{(1,-1,0)}$, all computable or well-approximable matrices with the connection problems described in Table~\ref{table:MOPs}. Subsequent differentiation is natural:
\[
\DD {\bf P}^{(k+1,0,k)}(x) = {\bf P}^{(k+2,0,k+1)}(x)D_{(k+1,0,k)}^{(k+2,0,k+1)},\quad\forall k\in\mathbb{N},
\]
where:
\begin{align*}
D_{(1,0,0)}^{(2,0,1)} & = R_{(0,0,1)}^{(2,0,1)} D_{(0,0,0)}^{(0,0,1)} (R_{(0,0,0)}^{(1,0,0)})^{-1},\\
D_{(k+1,0,k)}^{(k+2,0,k+1)} & = R_{(k+1,0,k)}^{(k+2,0,k+1)} D_{(k,0,k-1)}^{(k+1,0,k)} (R_{(k,0,k-1)}^{(k+1,0,k)})^{-1},\quad\forall k\in\mathbb{N}.
\end{align*}
All of these connection problems are banded with a minimal bandwidth in the sense that they are independent of $v(x)$.

Eq.~\eqref{eq:classicalderivative} offers two more practical interpretations. Firstly, the transpose of the differentiation matrix is the discretization of the weighted negative derivative:
\[
(-\mathcal{D})\left[\sigma(x)w(x){\bf P}'(x)\right] = w(x){\bf P}(x) (D_P^{P'})^\top.
\]
Secondly, an indefinite integral of ${\bf P}'(x)$ is discretized by the Moore--Penrose pseudoinverse~\cite{Golub-Van-Loan-13}:
\[
\int^x {\bf P}'(x){\rm\,d}x = {\bf P}(x)(D_P^{P'})^+.
\]
Since $D_P^{P'}$ has full row rank and only one nontrivial band, $(D_P^{P'})^+$ is a right inverse that is particularly easy to compute: it too only has one nontrivial band, and $(D_P^{P'})^+_{n+1,n} = (D_P^{P'})^{-1}_{n,n+1}$. Both of these properties carry over to the more general setting of Theorem~\ref{theorem:bandedderivatives}.
\begin{corollary}\label{corollary:bandedweightedderivativesandintegrals}
Consider ${\bf P}^{(\bs{\alpha},\beta)}(x)$ to be the same as in Theorem~\ref{theorem:bandedderivatives}.
\begin{enumerate}
\item The weighted negative derivative is:
\begin{align*}
& (-\mathcal{D})\left[w^{(\bs{\alpha}+\lambda(\bs{\alpha}),\beta+1)}(x){\bf P}^{(\bs{\alpha}+\lambda(\bs{\alpha}),\beta+1)}(x)\right]\\
& = w^{(\bs{\alpha},\beta)}(x){\bf P}^{(\bs{\alpha},\beta)}(x) (D_{(\bs{\alpha},\beta)}^{(\bs{\alpha}+\lambda(\bs{\alpha}),\beta+1)})^\top,
\end{align*}
\item and an indefinite integral is:
\[
\int^x {\bf P}^{(\bs{\alpha}+\lambda(\bs{\alpha}),\beta+1)}(x){\rm\,d}x = {\bf P}^{(\bs{\alpha},\beta)}(x)(D_{(\bs{\alpha},\beta)}^{(\bs{\alpha}+\lambda(\bs{\alpha}),\beta+1)})^+.
\]
\end{enumerate}
\end{corollary}
While $D_{(\bs{\alpha},\beta)}^{(\bs{\alpha}+\lambda(\bs{\alpha}),\beta+1)}$ in general has a larger upper bandwidth than $D_P^{P'}$, its full row rank enables direct computation of successive columns of its pseudoinverse, preserving the lower bandwidth of $1$. In particular, we may take the first row of $(D_{(\bs{\alpha},\beta)}^{(\bs{\alpha}+\lambda(\bs{\alpha}),\beta+1)})^+$ to be zero, and for the rest:
\[
(D_{(\bs{\alpha},\beta)}^{(\bs{\alpha}+\lambda(\bs{\alpha}),\beta+1)})^+_{2:n+1,1:n} = (D_{(\bs{\alpha},\beta)}^{(\bs{\alpha}+\lambda(\bs{\alpha}),\beta+1)})_{1:n,2:n+1}^{-1}.
\]
Combining Theorem~\ref{theorem:bandedderivatives} and Corollary~\ref{corollary:bandedweightedderivativesandintegrals} is now irresistible.
\begin{theorem}
Consider ${\bf P}^{(\bs{\alpha},\beta)}(x)$ to be the same as in Theorem~\ref{theorem:bandedderivatives}.
\begin{align*}
& (-\mathcal{D})\left[w^{(\bs{\alpha}+\lambda(\bs{\alpha}),\beta+1)}(x)\mathcal{D}\right]{\bf P}^{(\bs{\alpha},\beta)}(x)\\
& = w^{(\bs{\alpha},\beta)}(x){\bf P}^{(\bs{\alpha},\beta)}(x)(D_{(\bs{\alpha},\beta)}^{(\bs{\alpha}+\lambda(\bs{\alpha}),\beta+1)})^\top D_{(\bs{\alpha},\beta)}^{(\bs{\alpha}+\lambda(\bs{\alpha}),\beta+1)}.
\end{align*}
\end{theorem}

\section{Algorithms for infinite-dimensional matrix factorizations}\label{section:infinitealgorithms}

As mentioned above, there is an important distinction between Cholesky and $QR$ decompositions of infinite-dimensional matrices and their reverse Cholesky and $QL$ decompositions: Finite sections of infinite-dimensional Cholesky and $QR$ decompositions are computable while in the $QL$ and reverse Cholesky factorizations, matrix factorizations are employed that in general are not computable (with a finite number of operations) for infinite-dimensional matrices. In finite dimensions, a $QL$ factorization may proceed by orthogonal transformations designed to lower-triangularize the matrix in question. It follows that these orthogonal transformations must begin in the bottom right corner but in infinite dimensions, the bottom right corner is always out of reach, see \cite{Webb-Thesis-17}. A similar argument follows for the reverse Cholesky factorization. In finite dimensions, symmetric elementary transformations introduce zeros starting with the last row and column.

Since we only require the $n\times n$ principal section of the connection coefficients, we develop iterative methods to compute the approximate infinite-dimensional $QL$ and reverse Cholesky factorizations based on larger principal sections. Let $N > n$ and partition $V$ as follows:
\[
V = \begin{pmatrix} V_N & V_b\\ V_b^\top & V_\infty\end{pmatrix}.
\]
Here, $V_N$ is the $N\times N$ principal section of $V$. It inherits the symmetric positive-definite and banded structure from $V$. Next, $V_b$ is an $N\times \infty$ block with a rank bounded by the bandwidth of $V$ and sparsity pattern of a $b\times b$ lower-triangular matrix in the first $b$ columns and the last $b$ rows. Finally, $V_\infty$ is the trailing infinite-dimensional section. 

\subsection{Approximating the $QL$ factorization}

We compute the $QL$ factorization of $V_N$ and we wish to analyse the proximity of its $n\times n$ principal section to the same sized section of the infinite-dimensional $QL$ factorization of a nearby matrix to $V$; thus, we shall make a perturbative argument. Recall that $P_n$ is the canonical orthogonal projection of Eq.~\eqref{eq:canonicalorthogonalprojection}. We compute $V_N = Q_N L_N$ and consider:
\[
\begin{pmatrix} Q_N^\top\\ & I\end{pmatrix} V = \begin{pmatrix} L_N & Q_N^\top V_b\\ V_b^\top & V_\infty\end{pmatrix}.
\]
In the first $n$ rows, the orthogonal transformation $Q_N^\top$ has successfully lower triangularized the first $N$ columns of $V$. However, in the next $b$ columns, the entries $P_n Q_N^\top V_b$ interfere with a complete lower triangularization. Were they $0$, we would declare victory as the orthogonal lower triangularization would have succeeded. In general, this is impossible and yet we notice that in practice for $N\gg n$ these entries may be $2$-normwise small. Thus, we conclude with the following criterion for convergence of principal sections of the $QL$ factorization.

\begin{lemma}\label{lemma:QLperturbation}
Given $\epsilon>0$, if:
\[
\|P_nQ_N^\top V_b\|_2 < \epsilon \|Q_N^\top V_b\|_2 = \epsilon\|V_b\|_2,
\]
then there exists a $\tilde{V}_b$ such that:
\begin{align*}
P_n Q_N^\top\tilde{V}_b & = 0,\quad {\rm and}\\
\|V_b-\tilde{V}_b\|_2 & < \epsilon\|V_b\|_2.
\end{align*}
\end{lemma}
\begin{proof}
We choose $\tilde{V}_b = Q_N(I_N-P_n)Q_N^\top V_b$. Then:
\[
P_nQ_N^\top\tilde{V}_b = P_nQ_N^\top Q_N(I_N-P_n)Q_N^\top V_b  = P_n(I_N-P_n)Q_N^\top V_b = 0,
\]
and:
\begin{align*}
V_b - \tilde{V}_b & = \left[I_N - Q_N(I_N-P_n)Q_N^\top\right]V_b\\
& = Q_N P_n Q_N^\top V_b,
\end{align*}
and the result follows.
\end{proof}
Since $\|V_b\|_2 \le \|V\|_2$, Lemma~\ref{lemma:QLperturbation} suggests that introducing a $2$-normwise small relative perturbation in $V_b$ is sufficient to result in an $n$-row lower triangularization of a $\tilde{V}$ by the orthogonal transformation $Q_N^\top$. Moreover, as the perturbation is localized to the $V_b$ block, the statement is in fact stronger than:
\[
\|V-\tilde{V}\|_2 < \epsilon \|V\|_2.
\]
It may not be immediately obvious, but Lemma~\ref{lemma:QLperturbation} guarantees that for every $M>N$, if $\tilde{V}_M = Q_ML_M$, then the first $n$ rows of $L_M$ coincide with those of $L_N$, and the first $n$ columns of $Q_M$ are also the first $n$ columns of $Q_N$. Since $L_M$ and $L_N$ are lower-triangular matrices, the first $n\times n$ principal sections are the same. Similarly, the $N\times n$ principal sections of $Q_M$ and $Q_N$ are same. If $N$ is sufficiently large, then $\tilde{V}$ is $2$-normwise close to $V$ and we have computed the exact partial $QL$ factorization of a nearby matrix.

Our algorithm is as follows: given $\epsilon>0$, we begin by setting $N = 2n$, we compute $V_N = Q_NL_N$, and we double $N$ until we may invoke Lemma~\ref{lemma:QLperturbation} and continue with the finite-dimensional aspect of the computations involved in the $QL$ factorization. As there is no universal guarantee for this to occur, the software implementation of our algorithm issues a warning if $N$ reaches a huge maximum value.

\subsection{Approximating the reverse Cholesky factorization}

As above, we compute the reverse Cholesky factorization of $V_N$, and we make a perturbative argument to consider the utility of its $n\times n$ principal section. With $V_N = L_N^\top L_N$ in hand:
\[
\begin{pmatrix} L_N^{-\top}\\ & I\end{pmatrix} V = \begin{pmatrix} L_N & L_N^{-\top} V_b\\ V_b^\top & V_\infty\end{pmatrix},
\]
and we wish to consider how closely we have lower-triangularized the first $n$ rows.

\begin{lemma}\label{lemma:reverseCholeskyperturbation}
Given $\epsilon>0$, if:
\[
\|L_N^\top P_nL_N^{-\top} V_b\|_2 < \epsilon\|V_b\|_2,
\]
then there exists a $\tilde{V}_b$ such that:
\begin{align*}
P_n L_N^{-\top}\tilde{V}_b & = 0,\quad {\rm and}\\
\|V_b-\tilde{V}_b\|_2 & < \epsilon\|V_b\|_2.
\end{align*}
\end{lemma}
\begin{proof}
We choose $\tilde{V}_b = L_N^\top(I_N-P_n)L_N^{-\top} V_b$. Then:
\[
P_nL_N^{-\top}\tilde{V}_b = P_nL_N^{-\top} L_N^\top(I_N-P_n)L_N^{-\top} V_b  = P_n(I_N-P_n)L_N^{-\top} V_b = 0,
\]
and:
\begin{align*}
V_b - \tilde{V}_b & = \left[I_N - L_N^\top(I_N-P_n)L_N^{-\top}\right]V_b\\
& = L_N^\top P_n L_N^{-\top} V_b,
\end{align*}
and the result follows.
\end{proof}
Since $\|V_b\|_2 \le \|V\|_2$, Lemma~\ref{lemma:reverseCholeskyperturbation} suggests that introducing a $2$-normwise small relative perturbation in $V_b$ is sufficient to result in an $n$-row lower triangularization of $V$ by $L_N^{-\top}$. Our algorithm is as follows: given $\epsilon>0$, we begin by setting $N = 2n$, we compute $V_N=L_N^\top L_N$, and we double $N$ until we may invoke Lemma~\ref{lemma:reverseCholeskyperturbation} and continue with the finite-dimensional aspect of the computations involved in the reverse Cholesky factorization.

\subsection{A model for infinite-dimensional matrix factorizations}

We consider a model problem of a rational modification describing a single simple pole off $[-1,1]$:
\[
v(x) = \alpha+2\beta x.
\]
If this modifies the (orthonormal) Chebyshev polynomials of the second-kind, then:
\[
V = \begin{pmatrix} \alpha & \beta\\ \beta & \alpha & \ddots\\ & \ddots & \ddots\end{pmatrix},
\]
and $|\alpha|> 2|\beta|$. This symmetric tridiagonal matrix is also Toeplitz~\cite[\S 4.7]{Golub-Van-Loan-13}, which allows us to draw from the literature on the $QL$ and Wiener--Hopf factorizations of Toeplitz matrices.

\subsubsection{$QL$ factorizations}

Let $G_i(c, s)$ denote the real Givens rotation matrix in the $e_ie_{i+1}$-plane embedded in the infinite-dimensional identity:
\[
G_i(c, s) = \begin{pmatrix} I_{i-1}\\ & c & s\\ & -s & c\\ & & & I\end{pmatrix},
\]
where $s^2+c^2=1$.
\begin{lemma}[Theorem 5.2.6 in~\cite{Webb-Thesis-17}]\label{lemma:infinitedimensionalQL}
If $\alpha > 2|\beta|$, the $QL$ factorization of $V$ exists and is given by:
\[
V = \prod_{i=1}^\infty G_i\left(c, s\right) \begin{pmatrix} \sqrt{\frac{(\alpha+\sqrt{\alpha^2-(2\beta)^2})\sqrt{\alpha^2-(2\beta)^2}}{2}}\\ 2\beta & \frac{\alpha+\sqrt{\alpha^2-(2\beta)^2}}{2}\\ \frac{2\beta^2}{\alpha+\sqrt{\alpha^2-(2\beta)^2}} & 2\beta & \ddots\\ & \ddots & \ddots\end{pmatrix},
\]
where:
\[
s = \frac{2\beta}{\alpha+\sqrt{\alpha^2-(2\beta)^2}},\quad{\rm and}\quad c = \sqrt{\frac{2\sqrt{\alpha^2-(2\beta)^2}}{\alpha+\sqrt{\alpha^2-(2\beta)^2}}},
\]
and the Givens rotations are applied to the lower triangular matrix in the order determined by their index; that is, the infinite product iteratively prepends Givens rotations on the left.
\end{lemma}
\begin{proof}
We parameterize some of the entries of $L$ by $s$ and $c$. We apply $G_1$ to $L$:
\begin{align*}
& G_1\left(c, s\right) \begin{pmatrix} \sqrt{\frac{(\alpha+\sqrt{\alpha^2-(2\beta)^2})\sqrt{\alpha^2-(2\beta)^2}}{2}}\\ 2\beta & \frac{\alpha+\sqrt{\alpha^2-(2\beta)^2}}{2}\\ s\beta & 2\beta & \ddots\\ & \ddots & \ddots\end{pmatrix}\\
& = \begin{pmatrix} \alpha & \beta\\ \beta\sqrt{\frac{2\sqrt{\alpha^2-(2\beta)^2}}{\alpha+\sqrt{\alpha^2-(2\beta)^2}}} & \sqrt{\frac{(\alpha+\sqrt{\alpha^2-(2\beta)^2})\sqrt{\alpha^2-(2\beta)^2}}{2}}\\ s\beta & 2\beta & \frac{\alpha+\sqrt{\alpha^2-(2\beta)^2}}{2}\\ & \ddots & \ddots & \ddots\end{pmatrix},\\
& = \begin{pmatrix} \alpha & \beta\\ c\beta & \sqrt{\frac{(\alpha+\sqrt{\alpha^2-(2\beta)^2})\sqrt{\alpha^2-(2\beta)^2}}{2}}\\ s\beta & 2\beta & \frac{\alpha+\sqrt{\alpha^2-(2\beta)^2}}{2}\\ & \ddots & \ddots & \ddots\end{pmatrix}.
\end{align*}
Then, applying $G_2$ to the result, we find:
\begin{align*}
& G_2\left(c, s\right)\begin{pmatrix} \alpha & \beta\\ c\beta & \sqrt{\frac{(\alpha+\sqrt{\alpha^2-(2\beta)^2})\sqrt{\alpha^2-(2\beta)^2}}{2}}\\ s\beta & 2\beta & \frac{\alpha+\sqrt{\alpha^2-(2\beta)^2}}{2}\\ & \ddots & \ddots & \ddots\end{pmatrix}\\
& = \begin{pmatrix} \alpha & \beta\\ \beta & \alpha & \beta\\ 0 & c\beta & \sqrt{\frac{(\alpha+\sqrt{\alpha^2-(2\beta)^2})\sqrt{\alpha^2-(2\beta)^2}}{2}}\\ & s\beta & 2\beta & \frac{\alpha+\sqrt{\alpha^2-(2\beta)^2}}{2}\\ & & \ddots & \ddots & \ddots\end{pmatrix}.
\end{align*}
We have now shown that the first two rows of $G_2G_1L$ are equal to the first two rows of $V$, and since the $2:\infty \times 2:\infty$ section of $G_2G_1L$ is equal to $G_1L$, the full factorization follows by induction.
\end{proof}
If $\alpha = 2|\beta|$, then the infinite product of Givens rotations with $|s|=1$ and $c=0$ is an isometry but it is not orthogonal as every entry in the first column is zero.

Next, we compare this result to the computable factorization $V_N = Q_NL_N$. 

\begin{lemma}\label{lemma:finitedimensionalQL}
If $\alpha>2\beta>0$, the $QL$ factorization of $V_N$ is given by:
\[
V_N = \prod_{i=1}^{N-1}G_i(c_i, s_i) L_N,
\]
where:
\begin{equation}\label{eq:structuredsinesandcosines}
s_{N-k} = \sqrt{\frac{\displaystyle\sum_{j=1}^k a_j^2}{\displaystyle\sum_{j=1}^{k+1}a_j^2}}, \quad{\rm and}\quad c_{N-k} = \sqrt{\frac{\displaystyle a_{k+1}^2}{\displaystyle\sum_{j=1}^{k+1}a_j^2}},
\end{equation}
where:
\begin{equation}\label{eq:geometricaj}
a_j = \frac{\beta^2}{\sqrt{\alpha^2-(2\beta)^2}}\left[\left(\frac{\alpha+\sqrt{\alpha^2-(2\beta)^2}}{2\beta}\right)^j - \left(\frac{2\beta}{\alpha+\sqrt{\alpha^2-(2\beta)^2}}\right)^j\right],
\end{equation}
and where:
\begin{align}
(L_N)_{1,1} & = c_1\alpha-s_1c_2\beta,\label{eq:LN11}\\
(L_N)_{k,k} & = \sqrt{(c_k\alpha-s_kc_{k+1}\beta)^2+\beta^2},\label{eq:LNmaindiag}\quad k>1,\\
(L_N)_{k+1,k} & = s_k\alpha+c_kc_{k+1}\beta,\label{eq:LNfirstsub}\\
(L_N)_{k+2,k} & = s_{k+1}\beta.\label{eq:LNsecondsub}
\end{align}
\end{lemma}
\begin{proof}
As $V_N$ is finite-dimensional, we may begin in the bottom right corner. As $V_N$ is tridiagonal, we focus on the six entries that change for each Givens rotation. Note that:
\[
s_N = 0,\quad{\rm and}\quad c_N = 1.
\]
Then the next sine and cosine are determined by $G_{N-1}(c_{N-1}, s_{N-1})^\top V_N$ introducing a zero in the $N-1,N$ position:
\begin{align*}
\begin{pmatrix} c_{N-1} & -s_{N-1}\\ s_{N-1} & c_{N-1}\end{pmatrix}\begin{pmatrix} \beta & \alpha & \beta\\ 0 & \beta & \alpha\end{pmatrix} & = \begin{pmatrix} c_{N-1}\beta & c_{N-1}\alpha-s_{N-1}\beta & 0\\ s_{N-1}\beta & s_{N-1}\alpha+c_{N-1}\beta & \sqrt{\alpha^2+\beta^2}\end{pmatrix},\\
& = \begin{pmatrix} c_{N-1}\beta & c_{N-1}\alpha-s_{N-1}c_N\beta & 0\\ s_{N-1}\beta & s_{N-1}\alpha+c_{N-1}c_N\beta & \sqrt{\alpha^2+\beta^2}\end{pmatrix},
\end{align*}
or:
\[
s_{N-1} = \frac{\beta}{\sqrt{\alpha^2+\beta^2}},\quad{\rm and}\quad c_{N-1} = \frac{\alpha}{\sqrt{\alpha^2+\beta^2}}.
\]
Introducing another zero in the $N-2,N-1$ position involves the computation $G_{N-2}(c_{N-2}, s_{N-2})^\top G_{N-1}(c_{N-1}, s_{N-1})^\top V_N$:
\begin{align*}
& \begin{pmatrix} c_{N-2} & -s_{N-2}\\ s_{N-2} & c_{N-2}\end{pmatrix}\begin{pmatrix} \beta & \alpha & \beta\\ 0 & c_{N-1}\beta & c_{N-1}\alpha-s_{N-1}c_N\beta\end{pmatrix}\\
& = \begin{pmatrix} c_{N-2}\beta & c_{N-2}\alpha-s_{N-2}c_{N-1}\beta & 0\\ s_{N-2}\beta & s_{N-2}\alpha+c_{N-2}c_{N-1}\beta & \sqrt{(c_{N-1}\alpha-s_{N-1}c_N\beta)^2+\beta^2}\end{pmatrix},
\end{align*}
or:
\begin{align*}
s_{N-2} & = \frac{\beta}{\sqrt{(c_{N-1}\alpha-s_{N-1}c_N\beta)^2+\beta^2}},\quad{\rm and}\\
c_{N-2} & = \frac{c_{N-1}\alpha-s_{N-1}c_N\beta}{\sqrt{(c_{N-1}\alpha-s_{N-1}c_N\beta)^2+\beta^2}}.
\end{align*}
After $k$ steps, $G_{N-k}(c_{N-k}, s_{N-k})^\top \cdots G_{N-1}(c_{N-1}, s_{N-1})^\top V_N$, the sine and cosine are determined by two nonlinear recurrence relations:
\begin{align}
s_{N-k}^2 & = \frac{\beta^2}{(c_{N-k+1}\alpha-s_{N-k+1}c_{N-k+2}\beta)^2+\beta^2},\quad{\rm and}\label{eq:sinerecurrence}\\
c_{N-k}^2 & = \frac{(c_{N-k+1}\alpha-s_{N-k+1}c_{N-k+2}\beta)^2}{(c_{N-k+1}\alpha-s_{N-k+1}c_{N-k+2}\beta)^2+\beta^2}.
\end{align}
Substituting Eqs.~\eqref{eq:structuredsinesandcosines} into Eq.~\eqref{eq:sinerecurrence}, we find:
\begin{align*}
\beta^2\sum_{j=1}^{k+1}a_j^2 & = \left[(c_{N-k+1}\alpha-s_{N-k+1}c_{N-k+2}\beta)^2+\beta^2\right]\sum_{j=1}^ka_j^2,\\
\beta^2a_{k+1}^2 & = (c_{N-k+1}\alpha-s_{N-k+1}c_{N-k+2}\beta)^2\sum_{j=1}^ka_j^2,\\
& = (a_k\alpha-a_{k-1}\beta)^2.
\end{align*}
Taking the positive square root on the both sides, the linear recurrence relation:
\[
a_{k+1}\beta = a_k\alpha-a_{k-1}\beta,\quad a_1 = \beta,\quad a_2 = \alpha,
\]
is solved by Eq.~\eqref{eq:geometricaj}. A careful bookkeeping of the above procedure provides the entries of $L_N$.
\end{proof}

\begin{theorem}\label{theorem:QLconvergence}
Let $\alpha>2\beta>0$ and let:
\[
\rho = \frac{\alpha+\sqrt{\alpha^2-(2\beta)^2}}{2\beta}>1.
\]
If:
\[
N > n + \log_\rho\left(\frac{\rho-\rho^{-1}}{2\epsilon} + \sqrt{\left(\frac{\rho-\rho^{-1}}{2\epsilon}\right)^2+1}\right)-1,
\]
then:
\[
\left\| P_nL - P_n\begin{pmatrix} L_N & 0\\ 0 & 0\end{pmatrix}\right\|_2 < \epsilon\|L\|_2.
\]
\end{theorem}
\begin{proof}
We wish to invoke Lemma~\ref{lemma:QLperturbation} with the Givens rotations from Lemma~\ref{lemma:finitedimensionalQL}. Then:
\[
P_nQ_N^\top V_b = P_n G_1^\top\cdots G_{N-1}^\top \beta e_N.
\]
As $P_n$ commutes with $G_i$ for $i<n$:
\begin{align*}
\|P_nQ_N^\top V_b\|_2 & = \|G_1^\top\cdots G_{n-1}^\top P_n G_n^\top \cdots G_{N-1}^\top \beta e_N \|_2,\\
& = \|P_n G_n^\top \cdots G_{N-1}^\top \beta e_N \|_2,\\
& = \|V_b\|_2 \prod_{k=1}^{N-n} |s_{N-k}|.
\end{align*}
Then, by Lemma~\ref{lemma:finitedimensionalQL}:
\begin{align*}
\prod_{k=1}^{N-n} |s_{N-k}| & = \prod_{k=1}^{N-n} \sqrt{\frac{\displaystyle\sum_{j=1}^k a_j^2}{\displaystyle\sum_{j=1}^{k+1}a_j^2}},\\
& = \sqrt{\frac{a_1^2}{\displaystyle\sum_{j=1}^{N-n+1}a_j^2}} < \sqrt{\frac{a_1^2}{a_{N-n+1}^2}},\\
& = \frac{a_1}{a_{N-n+1}} = \frac{\rho-\rho^{-1}}{\rho^{N-n+1}-\rho^{n-N-1}} = \frac{\rho-\rho^{-1}}{2\sinh[(N-n+1)\log\rho]}.
\end{align*}
By our lower bound on $N$, we guarantee that $\|P_nQ_N^\top V_b\|_2 < \epsilon\|V_b\|_2$, and it follows from Lemma~\eqref{lemma:QLperturbation} at the subsequent discussion that the first $n\times n$ principal section of $L_N$ is $2$-normwise close to the same principal section of $L$.
\end{proof}

\begin{remark}
Consider the Bernstein ellipse:
\[
{\cal E}_\rho = \left\{ z\in\mathbb{C} : z = \frac{\rho e^{{\rm i}\theta} + \rho^{-1} e^{-{\rm i}\theta}}{2},\quad \theta\in[0,2\pi),\quad \rho\ge1\right\}.
\]
Then the real pole of $r(x)$, or the root of $v(x)$, is located at $-\frac{\alpha}{2\beta}$, corresponding to the Bernstein ellipse parameter $\rho = \frac{\alpha+\sqrt{\alpha^2-(2\beta)^2}}{2\beta}$. It is no surprise that the rate of exponential convergence of the finite-dimensional sines to their asymptotic limits is precisely equal to the logarithm of the Bernstein ellipse parameter. But we feel the physical connection is important: the closer the pole, the slower the convergence of finite-dimensional $QL$ to the infinite-dimensional.
\end{remark}

\subsubsection{Reverse Cholesky factorizations}

\begin{lemma}\label{lemma:infinitedimensionalreverseCholesky}
If $\alpha > 2|\beta|$, the reverse Cholesky factorization of $V$ is given by:
\[
V = \begin{pmatrix} l_d & l_o\\ & l_d & \ddots\\ & & \ddots\end{pmatrix}\begin{pmatrix} l_d\\ l_o & l_d\\ & \ddots & \ddots\end{pmatrix},
\]
where:
\[
l_d = \sqrt{\frac{\alpha+\sqrt{\alpha^2-(2\beta)^2}}{2}},\quad{\rm and}\quad l_o = \beta\sqrt{\frac{2}{\alpha+\sqrt{\alpha^2-(2\beta)^2}}}.
\]
\end{lemma}
\begin{proof}
On main diagonal entries, we find:
\[
\alpha = l_d^2+l_o^2 = \frac{\alpha+\sqrt{\alpha^2-(2\beta)^2}}{2} + \frac{2\beta^2}{\alpha+\sqrt{\alpha^2-(2\beta)^2}},
\]
which is true. On sub- and super-diagonal entries, it is also true that:
\[
\beta = l_dl_o.
\]
\end{proof}
This lemma is an instance of the discrete Wiener--Hopf factorization~\cite{Bini-Gemignani-Meini-343-21-02,Bottcher-Silbermann-99}, which describes a $UL$ factorization of a Toeplitz operator resulting in Toeplitz factors where all the work revolves around factoring the Toeplitz symbol.

As we may expect, the finite-dimensional reverse Cholesky factorization of a Toeplitz matrix is more involved.

\begin{lemma}\label{lemma:finitedimensionalreverseCholesky}
If $\alpha>2|\beta|$, the reverse Cholesky factorization of $V_N$ is given by:
\[
V_N = L_N^\top L_N,
\]
where:
\begin{equation}\label{eq:reverseCholeskymaindiagfirstsub}
(L_N)_{N-k,N-k} = \sqrt{d_k},\quad{\rm and}\quad (L_N)_{N-k,N-k-1} = \frac{\beta}{\sqrt{d_k}},
\end{equation}
and where:
\begin{equation}
d_k = \frac{b_{k+1}}{b_k},
\end{equation}
where:
\begin{equation}\label{eq:ek}
b_k = \frac{1}{\sqrt{\alpha^2-(2\beta)^2}}\left[\left(\frac{\alpha+\sqrt{\alpha^2-(2\beta)^2}}{2}\right)^{k+1} - \left(\frac{2\beta^2}{\alpha+\sqrt{\alpha^2-(2\beta)^2}}\right)^{k+1}\right].
\end{equation}
\end{lemma}
\begin{proof}
By considering the product $L_N^\top L_N$, we recover the off-diagonal relationship in Eq.~\eqref{eq:reverseCholeskymaindiagfirstsub} and the nonlinear recurrence for the main diagonal:
\[
(L_N)_{N-k-1,N-k-1}^2 + \frac{\beta^2}{(L_N)_{N-k,N-k}^2} = \alpha,\quad{\rm and}\quad (L_N)_{N-k,N-k-1} = \frac{\beta}{(L_N)_{N-k,N-k}}.
\]
The recurrence relation is linearized by the definitions of $d_k$ and $b_k$, resulting in:
\[
b_{k+2} -\alpha b_{k+1}+\beta^2b_k = 0,\quad b_0 = 1,\quad b_1 = \alpha.
\]
This linear recurrence relation is solved by Eq.~\eqref{eq:ek}.
\end{proof}

We could replicate an analogous result to Theorem~\ref{theorem:QLconvergence}, but the estimation of $\| L_N^\top P_n L_n^{-\top} V_b\|_2$ to invoke Lemma~\ref{lemma:reverseCholeskyperturbation} is more technically complicated. Instead, we will directly show that the $n\times n$ principal sections of $L$ and $L_N$ are $2$-normwise close, provided that $N$ is sufficiently large.

\begin{theorem}
Let $\alpha>2\beta>0$ and let:
\[
\rho = \frac{\alpha+\sqrt{\alpha^2-(2\beta)^2}}{2\beta}>1.
\]
If:
\[
N > n + \log_\rho\left(\frac{2}{\epsilon}\sqrt{\frac{\beta}{\alpha+2\beta}}\right)-\frac{1}{2},
\]
then:
\[
\left\| P_nL - P_n\begin{pmatrix} L_N & 0\\ 0 & 0\end{pmatrix}\right\|_2 < \epsilon\|L\|_2.
\]
\end{theorem}
\begin{proof} As a reverse Cholesky factor, it follows that:
\[
\|L\|_2 = \sqrt{\|V\|_2} = \sqrt{\alpha+2\beta}.
\]
The largest singular value of a matrix is bounded above by the maximum over all the absolute row and column sums~\cite{Qi-56-105-84}. For our bidiagonal matrix, this means:
\begin{align*}
\left\| P_nL - P_n\begin{pmatrix} L_N & 0\\ 0 & 0\end{pmatrix}\right\|_2 & \le \max\left\{\max_{2\le k\le n} \left[|L_{k,k} - (L_N)_{k,k}| + |L_{k,k-1} - (L_N)_{k,k-1}|\right],\right.\\
&\qquad\qquad \left. \max_{1\le k\le n} \left[|L_{k,k} - (L_N)_{k,k}| + |L_{k+1,k} - (L_N)_{k+1,k}|\right]\right\}.
\end{align*}
Since $\sqrt{d_k}$ is monotonically decreasing to $l_d$, the absolute row sums are uniformly larger than the absolute column sums and for $n>1$, the maximum absolute row sum is attained by the last row:
\[
\left\| P_nL - P_n\begin{pmatrix} L_N & 0\\ 0 & 0\end{pmatrix}\right\|_2 \le |L_{n,n} - (L_N)_{n,n}| + |L_{n,n-1} - (L_N)_{n,n-1}|.
\]
On the main diagonal, we have:
\begin{align*}
|l_d - \sqrt{d_k}| & = \left| \sqrt{\beta\rho} - \sqrt{\beta\frac{\rho^{k+2}-\rho^{-k-2}}{\rho^{k+1}-\rho^{-k-1}}}\right|,\\
& = \sqrt{\beta}\left| \sqrt{\rho} - \sqrt{\frac{\rho^{k+2} - \rho^{-k}}{\rho^{k+1}-\rho^{-k-1}} + \frac{\rho^{-k}-\rho^{-k-2}}{\rho^{k+1}-\rho^{-k-1}}}\right|,\\
& = \sqrt{\beta}\left| \sqrt{\rho} - \sqrt{\rho + \rho^{-1-2k}\frac{1-\rho^{-2}}{1-\rho^{-2k-2}}}\right|.
\end{align*}
Since $|\sqrt{x} -\sqrt{x+y}| < \sqrt{y}$ for $x,y>0$, it follows that:
\[
|l_d - \sqrt{d_k}| < \sqrt{\beta}\sqrt{\rho^{-1-2k}\frac{1-\rho^{-2}}{1-\rho^{-2k-2}}} < \sqrt{\beta}\rho^{-k-\frac{1}{2}}.
\]
On the sub-diagonal, we have the reciprocal difference:
\begin{align*}
\left|l_o - \frac{\beta}{\sqrt{d_k}}\right| & = \left|\sqrt{\frac{\beta}{\rho}} - \frac{\beta}{\sqrt{d_k}}\right|,\\
& = \sqrt{\beta}\left|\frac{1}{\sqrt{\rho}} - \frac{1}{\sqrt{\rho + \rho^{-1-2k}\frac{1-\rho^{-2}}{1-\rho^{-2k-2}}}}\right|,
\end{align*}
This time, we use $|\frac{1}{\sqrt{x}} - \frac{1}{\sqrt{x+y}}| < \sqrt{\frac{y}{x(x+y)}} < \frac{\sqrt{y}}{x} < \sqrt{y}$ for $x>1$ and $y>0$ to find:
\[
\left|l_o - \frac{\beta}{\sqrt{d_k}}\right| < \sqrt{\beta}\sqrt{\rho^{-1-2k}\frac{1-\rho^{-2}}{1-\rho^{-2k-2}}} < \sqrt{\beta}\rho^{-k-\frac{1}{2}}.
\]
Setting $k=N-n$, the inequality is attained.
\end{proof}
If $\beta<0$, a similar analysis reveals the same exponential convergence of the finite reverse Cholesky factor to its infinite-dimensional analogue.

\section{Applications and numerical experiments}\label{section:applicationsnumexp}

The free and open-source implementation of our algorithms is available in C at {\tt FastTransforms}~\cite{Slevinsky-GitHub-FastTransformsC}. Our numerical experiments are conducted on an iMac Pro (Early 2018) with a $2.3$ GHz Intel Xeon W-2191B with $128$ GB $2.67$ GHz DDR4 RAM. The library is compiled with {\tt -Ofast -march=native} optimization flags.

\subsection{Rapid modified Jacobi synthesis and analysis}

Consider the rational Jacobi weight modification:
\[
r(x;\gamma) = \frac{x^2+(500\gamma)^2}{[(x-\frac{1}{2})^2+\gamma^2]^2[(x+\frac{3}{4})^2+\gamma^2]}.
\]
We expand the corresponding orthonormal polynomials, $q_n^{(\alpha,\beta)}(x;\gamma)$, in Jacobi polynomials with the same parameters. Then, by a Jacobi--Chebyshev transform~\cite{Olver-Slevinsky-Townsend-29-573-20}, we further convert to first-kind Chebyshev polynomials, and rapidly synthesize the modified polynomials (and expansions) on a Chebyshev grid via the fast inverse discrete cosine transform (iDCT)~\cite{Frigo-Johnson-93-216-05}. This entire procedure requires:
\[
\underbrace{\mathcal{O}\left(n + \log_{\rho}\left[\frac{\rho-\rho^{-1}}{2\epsilon}+\sqrt{\left(\frac{\rho-\rho^{-1}}{2\epsilon}\right)^2+1}\right]\right)}_{\rm rational~modification~by~Theorem~\ref{theorem:QLconvergence}} + \underbrace{\mathcal{O}(n\log n\log\epsilon^{-1})}_{\rm Jacobi-Chebyshev} + \underbrace{\mathcal{O}(n\log n)}_{\rm iDCT},
\]
flops, where $\rho = \rho(\gamma)$ is the smallest Bernstein ellipse parameter for the four distinct poles of $r(x;\gamma)$. The left panel of Figure~\ref{fig:modifiedjacobi} illustrates one such polynomial and Szeg\H o's corresponding asymptotic envelope~\cite[Theorem 12.1.4]{Szego-75}.

With the connection problem in hand, we may also modify the Jacobi matrices for the construction of modified Gaussian quadrature rules. The right panel of Figure~\ref{fig:modifiedjacobi} illustrates how the $30$-point rule is modified for $10^{-4}<\gamma<10^2$; in particular, we note that some nodes converge toward the projection of the poles of $r(x;\gamma)$ on the real axis as $\gamma\to0$ and are dispersed by the appearance of roots of $r(x;\gamma)$ near origin.

\begin{figure}[htbp]
\begin{center}
\begin{tabular}{cc}
\includegraphics[width=0.465\textwidth]{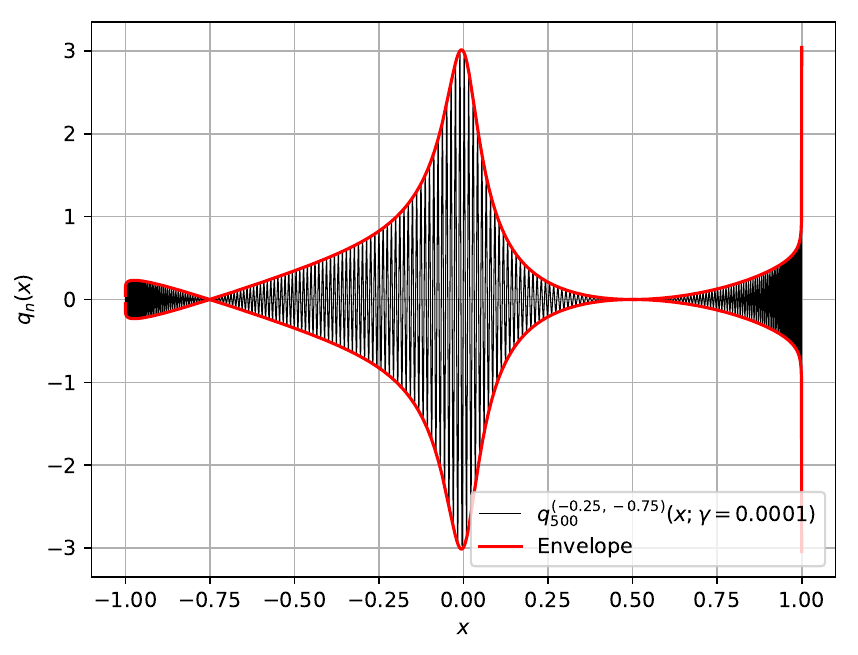}&
\includegraphics[width=0.465\textwidth]{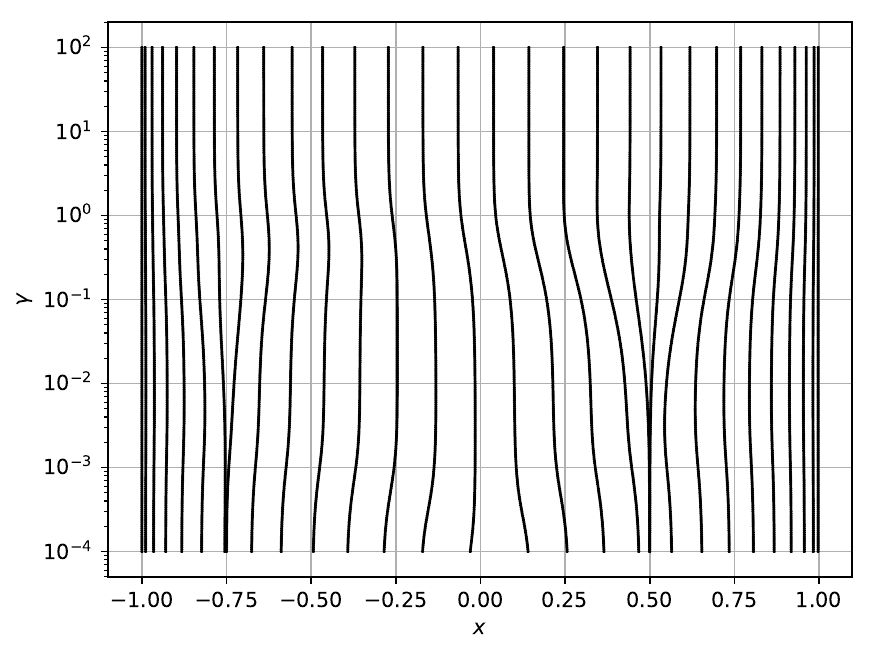}\\
\end{tabular}
\caption{Left: synthesis of $q_{500}^{(-0.25,-0.75)}(x;\gamma=0.0001)$ on a Chebyshev grid with the Szeg\H o envelope. Right: the nodes of the $30$-point modified Gaussian quadrature rule with $(\alpha,\beta) = (-0.25, -0.75)$.}
\label{fig:modifiedjacobi}
\end{center}
\end{figure}

Figure~\ref{fig:modifiedjacobierrortiming} illustrates the relative error and calculation times when working with degree-$n$ modified Jacobi polynomials. The left panel shows that the method is numerically stable. Notably in the right panel, the complexity of the factorization closely matches the result of Theorem~\ref{theorem:QLconvergence} for the simple pole modifying second-kind Chebyshev polynomials. This complexity includes a relatively large overhead at small degrees as $N\gg n$ such that it may accurately compute $2$-normwise nearby approximations to finite sections of the orthogonal and lower triangular factors of $V$. With the precomputation in hand, the execution times appear to produce relatively clean and predictable complexity plots.

\begin{figure}[htbp]
\begin{center}
\begin{tabular}{cc}
\includegraphics[width=0.465\textwidth]{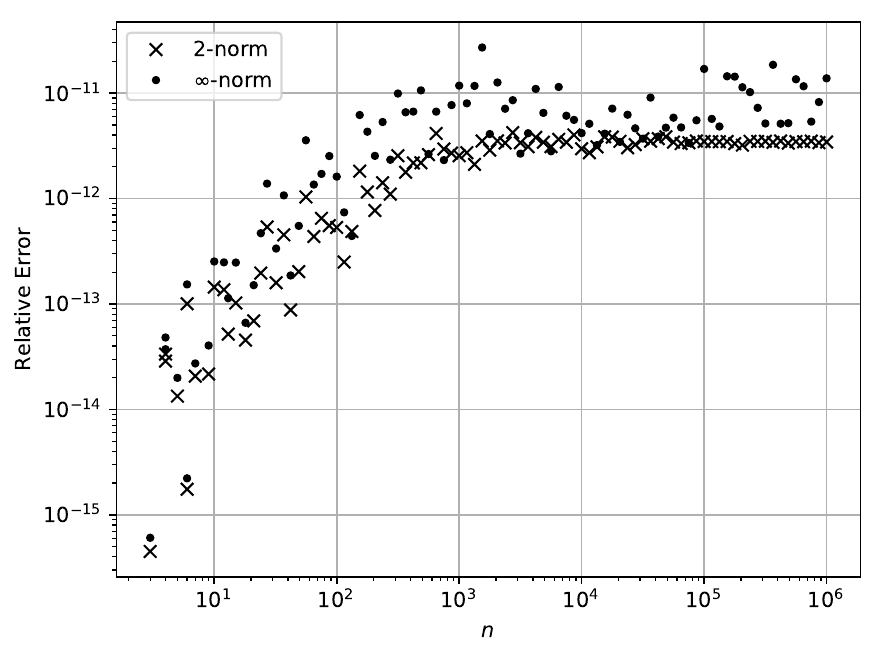}&
\includegraphics[width=0.465\textwidth]{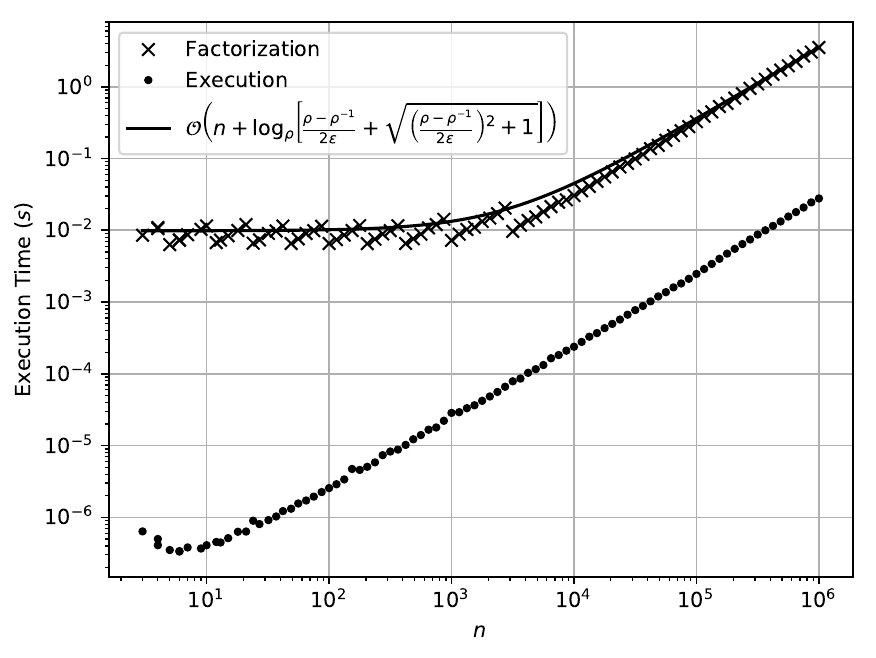}\\
\end{tabular}
\caption{Conversion of a degree-$n$ expansion in modified Jacobi polynomials $q_n^{(-0.25,-0.75)}(x;\gamma=0.01)$ with standard normally distributed pseudorandom coefficients to Jacobi polynomials with the same parameters. Left: $2$-norm and $\infty$-norm relative error in the forward and backward transformation. Right: precomputation and execution time as well as a complexity estimate based on Theorem~\ref{theorem:QLconvergence}.}
\label{fig:modifiedjacobierrortiming}
\end{center}
\end{figure}

\subsection{Modified Laguerre polynomials}

As alluded to in \S\ref{subsection:irrationalmeasuremodifications}, it would be impossible for a Laguerre polynomial series approximating $r(x) = \frac{x}{x+\gamma}$ to also well-approximate the infinite matrix $r(X_P)$. By the logical rational approximation, it is straightforward to consider the rationally modified Laguerre polynomials. Figure~\ref{fig:modifiedlaguerreerrortiming} illustrates the relative error and calculation times for transforming to and from these modified generalized Laguerre polynomials.

Notice that the relative error in forward-backward transformations appears to grow linearly with the truncation degree, which is markedly different from the asymptotically bounded relative error in the modified Jacobi polynomial transforms. We suspect this is due to the linear growth in the condition numbers of $U$ and $V$ that nearly cancel in $V^{-1}U$ and this inherent instability in the representation of a rational function as the quotient of two polynomials in particular on unbounded domains. We propose an avenue of future research that may address this issue in \S\ref{section:conclusion}.

For generalized Laguerre series, a {\em parabola} co-axial with the $x$-axis opening to the right with focus the origin and vertex chosen to maximize the region of analyticity is the analogue of a Bernstein ellipse for Jacobi series \cite{Szasz-Yeardley-8-621-58}. For the simple pole of $r(x)$ at $x=-\gamma$, the corresponding parabola is $y^2 = 4\gamma(x+\gamma)$, and generalized Laguerre series converge root-exponentially at the rate $\mathcal{O}(e^{-2\sqrt{\gamma n}})$ as $n\to\infty$. From this rate, we include in Figure~\ref{fig:modifiedlaguerreerrortiming} a graph estimating the complexity of the precomputation based on this region of analyticity.

\begin{figure}[htbp]
\begin{center}
\begin{tabular}{cc}
\includegraphics[width=0.465\textwidth]{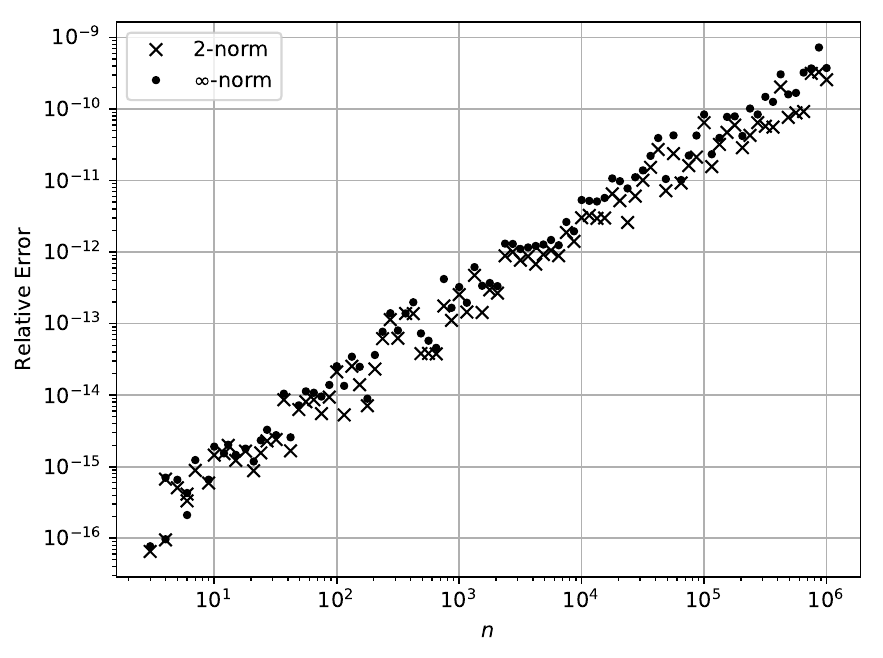}&
\includegraphics[width=0.465\textwidth]{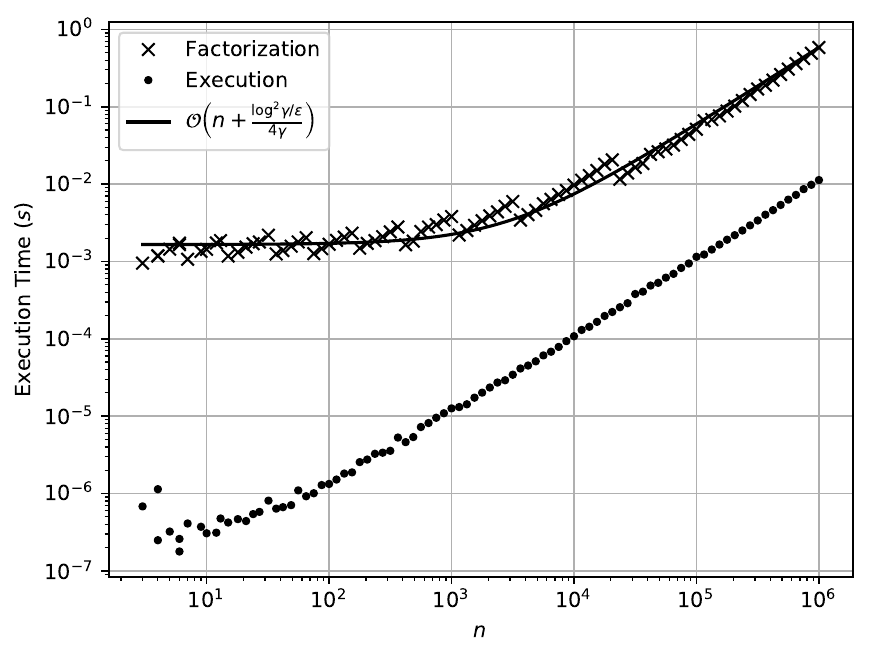}\\
\end{tabular}
\caption{Conversion of a degree-$n$ expansion in modified generalized Laguerre polynomials $q_n^{(0.25)}(x;\gamma=0.1)$ with standard normally distributed pseudorandom coefficients to generalized Laguerre polynomials with the same parameter. Left: $2$-norm and $\infty$-norm relative error in the forward and backward transformation. Right: precomputation and execution time as well as a complexity estimate inspired by Theorem~\ref{theorem:QLconvergence}.}
\label{fig:modifiedlaguerreerrortiming}
\end{center}
\end{figure}

\subsection{An orthonormal basis for $L^2(\mathbb{R})$}

Consider the problem of finding an orthonormal basis, ${\bf F}(x) = (f_0(x) \quad f_1(x) \quad f_2(x) \quad \cdots)$, for $L^2(\mathbb{R})$ with algebraic decay:
\[
\int_{-\infty}^\infty f_m(x)f_n(x){\rm\,d}x = \delta_{m,n},\quad |f_n(x)| = \Theta(x^{-\alpha}),\quad x\to\pm\infty,
\]
for~some $\frac{1}{2}<\alpha<\infty$.
Llewellyn Smith and Luca~\cite{Llewellyn-Smith-Luca-475-20190105-19} use the rational map $\displaystyle x = \frac{t}{1-t^2}$ to transform the problem to one on $(-1,1)$, where orthonormality reads:
\[
\int_{-1}^1 f_m(x(t)) f_n(x(t)) \frac{1+t^2}{(1-t^2)^2}{\rm\,d}t = \delta_{m,n}.
\]
Define $q_n(t)$ to be orthonormal polynomials in $L^2([-1,1], (1+t^2){\rm d}t)$. Of course, they are connected to the normalized Legendre polynomials by ${\bf P}(t) = {\bf Q}(t) R$ where $R$ is the upper-triangular Cholesky factor of $I+X_P^2$. It follows that:
\[
f_n(x) = (1-t^2)q_n(t) = \frac{2}{\sqrt{1+4x^2}+1} q_n\left(\frac{2x}{\sqrt{1+4x^2}+1}\right).
\]
An important property of an orthonormal basis on $L^2(\mathbb{R})$ is that differentiation, $\mathcal{D}$, is a skew-adjoint linear operator. When evolving time-dependent partial differential equations numerically, if the discretization of this operator is also skew-adjoint and sparse for any principal finite section, then fast methods exist for its spectral decomposition~\cite{Iserles-Webb-19-1191-19}, and the time evolution is stable.

We will now show that given the eigenvalue problem:
\[
\mathcal{D}u(x) = \lambda u(x), \qquad \lim_{x\to\pm\infty}u(x) = 0,
\]
there is a skew-definite and banded discretization in the orthonormal basis ${\bf F}(x)$:
\begin{align*}
\mathcal{D}{\bf F}(x)U = {\bf F}(x) D U & = {\bf F}(x) U \Lambda,\\
R^\top DRV & = R^\top RV\Lambda,
\end{align*}
and $U = RV$.

By changing coordinates:
\[
D = \int_\bbR {\bf F}(x)^\top \mathcal{D}{\bf F}(x){\rm\,d}x = \int_{-1}^1 (1-t^2){\bf Q}(t)^\top \mathcal{D}\left[(1-t^2){\bf Q}(t)\right]{\rm\,d}t,
\]
and the connection problem congruence transformation:
\[
R^\top D R = \int_{-1}^1 (1-t^2){\bf P}(t)^\top \mathcal{D}\left[(1-t^2){\bf P}(t)\right]{\rm\,d}t,
\]
uncovers banded sparsity on the right-hand side, by orthogonality of Legendre polynomials. Legendre polynomials are an instance of Jacobi polynomials, and by classical connections we find two formul\ae~for the right-hand side that also show the skew symmetry:
\[
R^\top D R = (I-X_P^2)(-D_P^{P'})^\top R_P^{P'} = (R_P^{P'})^\top D_P^{P'}(I-X_P^2).
\]

Another useful property for any basis is a uniform pointwise bound. We can show that:
\[
|f_n(x)| < \sqrt{\frac{2}{\pi}}\frac{2^{\frac{7}{4}}}{(1+4x^2)^{\frac{3}{8}}},\quad\forall x\in\mathbb{R}.
\]
By using line (1) in Table~\ref{table:MOPs}, $(1+t^2){\bf Q}(t) = {\bf P}(t)R^\top$, and noting that by symmetry $R$ is banded with only two nontrivial bands:
\begin{align*}
|f_n(x)| & = |(1-t^2)q_n(t)|,\\
& \le |(1-t^2)(1+t^2)q_n(t)|,\\
& = |(1-t^2)(p_n(t)R_{n,n}+p_{n+2}(t)R_{n,n+2})|.
\end{align*}
Recall the sharpened Bernstein inequality~\cite{Antonov-Holsevnikov-13-163-81}:
\[
(1-x^2)^{\frac{1}{4}}|P_n(x)| < \sqrt{\frac{2}{\pi}}\frac{1}{\sqrt{n+\frac{1}{2}}}, \quad x\in[-1,1].
\]
Then, since $p_n(t)$ are the orthonormal Legendre polynomials:
\begin{align*}
|f_n(x)| & \le |(1-t^2)^{\frac{3}{4}}|\left[|(1-t^2)^{\frac{1}{4}}p_n(t)||R_{n,n}| + |(1-t^2)^{\frac{1}{4}}p_{n+2}(t)||R_{n,n+2}|\right],\\
& < |(1-t^2)^{\frac{3}{4}}|\sqrt{\frac{2}{\pi}}\left(|R_{n,n}| + |R_{n,n+2}|\right).
\end{align*}
Now:
\begin{align*}
|R_{n,n}| + |R_{n,n+2}| & = \|R^\top e_n\|_1 \le \sqrt{2}\|R^\top e_n\|_2 \le \sqrt{2}\|R^\top\|_2,\\
& = \sqrt{2}\sqrt{\|R^\top R\|_2} = 2,\quad{\rm since}\quad \|R^\top R\|_2 = \sup_{t\in[-1,1]}|1+t^2| = 2.
\end{align*}

\subsection{Orthogonal polynomials on an annulus}\label{subsection:annulus}

Consider $P_n^{t,(\alpha,\beta,\gamma)}(x)$ to be orthonormal polynomials in $L^2([-1,1], (1-x)^\alpha(1+x)^\beta(t+x)^\gamma \dx)$, for parameter values $\{t>1, \alpha,\beta>-1, \gamma\in\mathbb{R}\}\cup\{t=1, \alpha,\beta+\gamma>-1\}$. If $\gamma\in\mathbb{Z}$, then $(t+x)^\gamma$ is either polynomial or rational and our algorithms to connect ${\bf P}^{t,(\alpha,\beta,\gamma)}(x)$ to the Jacobi polynomials ${\bf P}^{(\alpha,\beta)}(x)$ are immediately applicable. For $\gamma\in\mathbb{R}\setminus\mathbb{Z}$ and $t>1$, $(t+x)^\gamma$ may be well-approximated by polynomials and rationals alike.

Real orthonormal annulus polynomials in:
\[
L^2(\{(r,\theta) : \rho < r < 1, 0 < \theta < 2\pi\},  r^{2\gamma+1}(r^2-\rho^2)^\alpha(1-r^2)^\beta{\rm\,d} r{\rm\,d}\theta),
\]
are:
\begin{align*}
Z_{\ell,m}^{\rho,(\alpha,\beta,\gamma)}(r,\theta) = & \sqrt{2} \left(\frac{2}{1-\rho^2}\right)^{\frac{|m|+\alpha+\beta+\gamma+1}{2}} r^{|m|}P_{\frac{\ell-|m|}{2}}^{\frac{1+\rho^2}{1-\rho^2},(\beta,\alpha,|m|+\gamma)}\left(\frac{2r^2-1-\rho^2}{1-\rho^2}\right)\\
&\times \sqrt{\frac{2-\delta_{m,0}}{2\pi}} \left\{\begin{array}{ccc} \cos(m\theta) & {\rm for} & m \ge 0,\\ \sin(|m|\theta) & {\rm for} & m < 0.\end{array}\right.
\end{align*}
For the purposes of synthesis and analysis on tensor-product grids on the annulus, we convert these polynomials to a direct sum of tensor-product bases through a sequence of orthogonal transformations, a generalization that has the same structure as Zernike polynomial transforms~\cite{Olver-Slevinsky-Townsend-29-573-20,Slevinsky-GitHub-FastTransformsC}. These are computed by considering that decrementing the order $m$ in steps of $2$ requires the computation of line (2) in Table~\ref{table:MOPs}:
\[
(t+x){\bf P}^{t,(\alpha,\beta,\gamma+m+2)}(x) = {\bf P}^{t,(\alpha,\beta,\gamma+m)}(x)Q_{t,(\alpha,\beta,\gamma+m+2)}^{t,(\alpha,\beta,\gamma+m)}.
\]
Moreover, the similarity transformation between the Jacobi matrices for ${\bf P}^{t,(\alpha,\beta,\gamma+m)}(x)$ and ${\bf P}^{t,(\alpha,\beta,\gamma+m+2)}(x)$ requires the upper-triangular factor in line (2) in Table~\ref{table:MOPs}:
\[
{\bf P}^{t,(\alpha,\beta,\gamma+m)}(x) = {\bf P}^{t,(\alpha,\beta,\gamma+m+2)}(x)R_{t,(\alpha,\beta,\gamma+m)}^{t,(\alpha,\beta,\gamma+m+2)},
\]
enabling:
\[
X_{t,(\alpha,\beta,\gamma+m+2)} = R_{t,(\alpha,\beta,\gamma+m)}^{t,(\alpha,\beta,\gamma+m+2)} X_{t,(\alpha,\beta,\gamma+m)} (R_{t,(\alpha,\beta,\gamma+m)}^{t,(\alpha,\beta,\gamma+m+2)})^{-1}.
\]
Both factors may be computed simultaneously from:
\[
tI+X_{t,(\alpha,\beta,\gamma+m)} = Q_{t,(\alpha,\beta,\gamma+m+2)}^{t,(\alpha,\beta,\gamma+m)}R_{t,(\alpha,\beta,\gamma+m)}^{t,(\alpha,\beta,\gamma+m+2)}.
\]

\subsection{Orthogonal polynomials on a spherical band}\label{subsection:sphericalband}

For multi-indices $\bs{t}$ and $\bs{\alpha}$, consider $P_n^{\bs{t},(\bs{\alpha})}(x)$ to be orthonormal polynomials in:
\[
L^2([-1,1], (t_1-x)^{\alpha_1}(1-x)^{\alpha_2}(1+x)^{\alpha_3}(t_2+x)^{\alpha_4}\dx).
\]
On a spherical band:
\[
\mathbb{S}_{\bs{\theta}}^2 = \{(\theta,\varphi) : 0 < \theta_1 < \theta < \theta_2 < \pi, 0 < \varphi < 2\pi\},
\]
with weight:
\[
w^{\bs{\theta},(\bs{\alpha})}(\theta) = (2\sin^2\tfrac{\theta}{2})^{\alpha_1}(\cos\theta_1-\cos\theta)^{\alpha_2}(\cos\theta-\cos\theta_2)^{\alpha_3}(2\cos^2\tfrac{\theta}{2})^{\alpha_4},
\]
real orthonormal spherical band polynomials in $L^2(\mathbb{S}_{\bs{\theta}}^2, w^{\bs{\theta},(\bs{\alpha})}(\theta)\sin\theta{\rm\,d}\theta{\rm\,d}\varphi)$ are:
\begin{align*}
Y_{\ell,m}^{\bs{\theta},(\bs{\alpha})}(\theta,\varphi) = & \left(\frac{2}{\cos\theta_1-\cos\theta_2}\right)^{\frac{|\bs{\alpha}|+2|m|+1}{2}} \sin^{|m|}\!\theta P_{\ell-|m|}^{\bs{t},(\bs{\alpha}_{|m|})}\left(\frac{2\cos\theta-\cos\theta_1-\cos\theta_2}{\cos\theta_1-\cos\theta_2}\right)\\
&\times \sqrt{\frac{2-\delta_{m,0}}{2\pi}} \left\{\begin{array}{ccc} \cos(m\theta) & {\rm for} & m \ge 0,\\ \sin(|m|\theta) & {\rm for} & m < 0,\end{array}\right.
\end{align*}
where:
\[
\bs{t} = (t_1, t_2)^\top = \left(\frac{2-\cos\theta_1-\cos\theta_2}{\cos\theta_1-\cos\theta_2}, \frac{2+\cos\theta_1+\cos\theta_2}{\cos\theta_1-\cos\theta_2}\right)^\top,
\]
and $\bs{\alpha}_m = (\alpha_1+m, \alpha_2, \alpha_3, \alpha_4+m)^\top$ and $|\bs{\alpha}| = \alpha_1+\alpha_2+\alpha_3+\alpha_4$.

For the purposes of synthesis and analysis on tensor-product grids on the spherical band, we convert these polynomials to a direct sum of tensor-product bases through a sequence of orthogonal transformations, a generalization that has the same structure as spherical harmonic transforms~\cite{Slevinsky-47-585-19,Slevinsky-GitHub-FastTransformsC}. These are computed by considering that decrementing the order $m$ in steps of $2$ requires the computation of line (2) in Table~\ref{table:MOPs}:
\[
(t_1-x)(t_2+x){\bf P}^{\bs{t},(\bs{\alpha}_{m+2})}(x) = {\bf P}^{\bs{t},(\bs{\alpha}_m)}(x)Q_{\bs{t},(\bs{\alpha}_{m+2})}^{\bs{t},(\bs{\alpha}_m)}.
\]
Moreover, the similarity transformation between the Jacobi matrices for ${\bf P}^{\bs{t},(\bs{\alpha}_m)}(x)$ and ${\bf P}^{\bs{t},(\bs{\alpha}_{m+2})}(x)$ requires the upper-triangular factor in line (2) in Table~\ref{table:MOPs}:
\[
{\bf P}^{\bs{t},(\bs{\alpha}_m)}(x) = {\bf P}^{\bs{t},(\bs{\alpha}_{m+2})}(x)R_{\bs{t},(\bs{\alpha}_m)}^{\bs{t},(\bs{\alpha}_{m+2})},
\]
enabling:
\[
X_{\bs{t},(\bs{\alpha}_{m+2})} = R_{\bs{t},(\bs{\alpha}_m)}^{\bs{t},(\bs{\alpha}_{m+2})} X_{\bs{t},(\bs{\alpha}_m)} (R_{\bs{t},(\bs{\alpha}_m)}^{\bs{t},(\bs{\alpha}_{m+2})})^{-1}.
\]
Both factors may be computed simultaneously from:
\[
(t_1I-X_{\bs{t},(\bs{\alpha}_m)})(t_2I+X_{\bs{t},(\bs{\alpha}_m)}) = Q_{\bs{t},(\bs{\alpha}_{m+2})}^{\bs{t},(\bs{\alpha}_m)}R_{\bs{t},(\bs{\alpha}_m)}^{\bs{t},(\bs{\alpha}_{m+2})}.
\]

\begin{remark}
The orthogonal structures of an annular sector and a spherical quadrangle follow naturally by replacing the Fourier modes with orthonormal polynomials on an arc.
\end{remark}

\section{Conclusions \& Future Directions}\label{section:conclusion}

A straightforward generalization of this work is to develop efficient algorithms in the multivariate setting. The Koornwinder construction~\cite{Koornwinder-435-75} provides bivariate analogues of the classical orthogonal polynomials, and other multivariate orthogonal polynomials are known in greater than two dimensions~\cite{Dunkl-Xu-14}. Bivariate polynomial and rational measure modifications can extend these polynomials to new and more interesting measures, where polynomial modifications result in banded-block-banded matrices: block-banded matrices whose blocks are also banded. A challenge is to explore improving the complexity of banded-block-banded matrix factorizations. In higher dimensions, symmetric Jacobi matrices of Koornwinder-like constructions are tridiagonal-block-banded~\cite{Kowalski-13-309-82,Xu-62-687-94,Xu-342-855-94}; hence, the ``bandwidths'' of polynomial weight modifications grows with the degree, resulting in an $\mathcal{O}(n^4)$ Cholesky factorization with $\mathcal{O}(n^3)$ nonzero entries in two dimensions. These complexities stand in contrast to the linear complexities in the univariate setting.

Another logical extension is to consider the full semi-classical setting, where the weight function $w(x)$ satisfies a first-order linear homogeneous differential equation with polynomial coefficients $a(x)$ and $b(x)$:
\[
\DD(aw) = bw.
\]
In this Pearson differential equation~\eqref{eq:Pearson}, the restrictions on the polynomial degrees are removed. Interesting connections between semi-classical orthogonal polynomials and the Painlev\'e transcendents~\cite{Magnus-57-215-95} have been discovered. While this case may not exhibit a visible sparsity in the connection problem, such as bandedness, it would be worthwhile exploring methods to formulate the connection coefficients as solutions to matrix equations.

Compared with previous approaches to rational measure modifications~\cite{Uvarov-9-1253-69,Price-16-999-79,Skrzipek-41-331-92}, where the polynomials are factored as a product of degree-$1$ or $2$ polynomials, our infinite-dimensional matrix factorizations leave $u(x)$ and $v(x)$ whole. This allows the more numerically stable use of orthogonal polynomial expansions of these variable coefficients, avoids the potentially unstable computation of their roots, and solves the connection problem in one step. Representing rational functions as a ratio of polynomials is no panacea. For example, the Newman rational~\cite{Newman-11-11-64} approximating $|x|$ creates numerator and denominator polynomials with large dynamic ranges. Hence, by orthogonal polynomial expansion, absolutely small negative parts of $v(X_P)$ are introduced numerically that lead to the nonexistence of $QL$ and reverse Cholesky factorizations. This example and others inducing Froissart doublets have led to the construction of an even more stable representation of rational functions that has been used to develop the so-called AAA algorithm~\cite{Nakatsukasa-Sete-Trefethen-40-A1494-18}. In AAA, a rational function is represented as a ratio of {\em rationals}:
\[
r(x) = \frac{\displaystyle \sum_{k=0}^n\frac{w_kf_k}{x-x_k}}{\displaystyle \sum_{k=0}^n\frac{w_k}{x-x_k}}.
\]
This definition is augmented by $\lim_{x\to x_k}r(x) = f_k$ and the constants $x_k$ and $w_k$ are the parameters at one's disposal to range over all type-$(n,n)$ rational functions. In AAA, the support points $x_k$ are greedily chosen from among a set of user-supplied sample points in the approximation domain. Consequently, both the numerator and denominator rationals have singularities on the approximation domain that cancel but which are a problem when considering (approximately) evaluating $r(X_P)$. A partial fraction decomposition could reveal the poles as distinct from the support points and the connection coefficients may be recovered by Cholesky factorization of a sum of invertible symmetric tridiagonal matrices~\cite[\S 7.2.1]{Vandebril-Van-Barel-Mastronardi-1-08}.

A final avenue of future work is to investigate the structure of the connection coefficients with piecewise-defined polynomial and rational measure modifications. Relatedly, we note that $\mathcal{O}(n^2)$ algorithms exist~\cite{Elhay-Golub-Kautsky-32-143-92} to construct Jacobi matrices for sums of weight functions. It remains to be seen whether or not these connection problems may be solved in reduced complexities for sums of weight functions. Here, we do not anticipate linear complexity; rather, we suspect quasi-optimal complexities based on hierarchical factorizations of the connection problem.

\section*{Acknowledgments}

We thank Mioara Joldes and Ioannis Papadopoulos for interesting discussions regarding this work and extensions, Marcus Webb for crucial guidance on infinite-dimensional factorizations, Yuji Nakatsukasa and Matthew Colbrook for insightful comments on using nonstandard inner product Householder approaches for infinite-dimensional matrix factorizations in the early phases of the project, as well as Jiajie Yao for valuable feedback.\\
TSG and SO were partially supported by an EPSRC grant (EP/T022132/1). RMS is supported by the Natural Sciences and Engineering Research Council of Canada, through a Discovery Grant (RGPIN-2017-05514). SO was also supported by a Leverhulme Trust Research Project Grant (RPG-2019-144). TSG was also supported by a PIMS-Simons postdoctoral fellowship, jointly funded by the Pacific Institute for the Mathematical Sciences (PIMS) and the Simons Foundation.

\bibliographystyle{siamplain}
\bibliography{refs}

\end{document}

%% file: ModifiedOPs_shared.tex
\usepackage{lipsum}
\usepackage{amsfonts}
\usepackage{graphicx}
\usepackage{epstopdf}
\usepackage{algorithmic}
\ifpdf
  \DeclareGraphicsExtensions{.eps,.pdf,.png,.jpg}
\else
  \DeclareGraphicsExtensions{.eps}
\fi

\newsiamremark{remark}{Remark}
\newsiamremark{hypothesis}{Hypothesis}
\crefname{hypothesis}{Hypothesis}{Hypotheses}
\newsiamthm{claim}{Claim}

\usepackage{pifont,latexsym,ifthen,calc}
\usepackage{amssymb,amsbsy,amsmath}
\usepackage[errorshow]{tracefnt}
\usepackage{hyperref}
\usepackage{enumerate}

\newcommand{\rnum}[1]{\MakeUppercase{\romannumeral #1}}
\newcommand{\RRone}{R_{\rnum{1}}}
\newcommand{\RRtwo}{R_{\rnum{1} \hspace{-2pt} \rnum{1}}}

\headers{Rational measure modifications and infinite matrix factorizations}{T. S. Gutleb, S. Olver, R. M. Slevinsky}

\title{Polynomial and rational measure modifications of orthogonal polynomials via infinite-dimensional banded matrix factorizations}

\author{Timon S. Gutleb\thanks{Department of Mathematics, University of British Columbia, Vancouver, BC, Canada (\email{gutleb@math.ubc.ca})}
\and Sheehan Olver\thanks{Department of Mathematics, Imperial College London, London, UK 
  (\email{s.olver@imperial.ac.uk})} \and Richard Mika\"el Slevinsky\thanks{Department of Mathematics, Winnipeg, Canada, (\email{Richard.Slevinsky@umanitoba.ca)}}}

\usepackage{amsopn}
\DeclareMathOperator{\diag}{diag}

\def\DD{\mathcal{D}}


%% file: ModifiedOPs.bbl
\begin{thebibliography}{10}

\bibitem{Akhiezer-65}
{\sc N.~I. Akhiezer}, {\em The classical moment problem and some related
  questions in analysis}, Oliver \& Boyd, 1965.

\bibitem{Antonov-Holsevnikov-13-163-81}
{\sc V.~A. Antonov and K.~V. Hol{\v s}evnikov}, {\em An estimate of the
  remainder in the expansion of the generating function for the {L}egendre
  polynomials ({G}eneralization and improvement of {B}ernstein's inequality)},
  Vestnik Leningrad Univ. Math., 13 (1981), pp.~163--166.

\bibitem{Aurentz-Slevinsky-410-109383-20}
{\sc J.~L. Aurentz and R.~M. Slevinsky}, {\em On symmetrizing the
  ultraspherical spectral method for self-adjoint problems}, J. Comp. Phys.,
  410 (2020), p.~109383.

\bibitem{Bini-Gemignani-Meini-343-21-02}
{\sc D.~A. Bini, L.~Gemignani, and B.~Meini}, {\em Computations with infinite
  {T}oeplitz matrices and polynomials}, Linear Algebra Appl., 343--344 (2002),
  pp.~21--61.

\bibitem{Bochner-29-730-29}
{\sc S.~Bochner}, {\em {\"U}ber {S}turm--{L}iouvillesche {P}olynomsysteme},
  Math. Z., 29 (1929), pp.~730--736.

\bibitem{Bottcher-Silbermann-99}
{\sc A.~B\"ottcher and B.~Silbermann}, {\em Introduction to Large Truncated
  Toeplitz Matrices}, Springer, New York, 1999.

\bibitem{Buhmann-Iserles-43-117-92}
{\sc M.~D. Buhmann and A.~Iserles}, {\em On orthogonal polynomials transformed
  by the ${QR}$ algorithm}, J. Comp. Appl. Math., 43 (1992), pp.~117--134.

\bibitem{Chui-Ward-Smith-5-1-82}
{\sc C.~K. Chui, J.~D. Ward, and P.~W. Smith}, {\em {C}holesky factorization of
  positive definite bi-infinite matrices}, Numerical Functional Analysis and
  Optimization, 5 (1982), pp.~1--20.

\bibitem{Clenshaw-9-118-55}
{\sc C.~W. Clenshaw}, {\em A note on the summation of {C}hebyshev series},
  Math. Comp., 9 (1955), pp.~118--120.

\bibitem{Colbrook-Hansen-143-17-19}
{\sc M.~J. Colbrook and A.~C. Hansen}, {\em On the infinite-dimensional {QR}
  algorithm}, Numer. Math., 143 (2019), pp.~17--83.

\bibitem{Deift-99}
{\sc P.~Deift}, {\em Orthogonal {P}olynomials and {R}andom {M}atrices: a
  {R}iemann--{H}ilbert {A}pproach}, AMS, 1999.

\bibitem{Deift-Li-Tomei-64-358-85}
{\sc P.~Deift, L.~C. Li, and C.~Tomei}, {\em {T}oda flows with infinitely many
  variables}, J. Funct. Anal., 64 (1985), pp.~358--402.

\bibitem{Dunkl-Xu-14}
{\sc C.~F. Dunkl and Y.~Xu}, {\em Orthogonal Polynomials of Several Variables},
  Encyclopedia of Mathematics and its Applications, Cambridge University Press,
  second~ed., 2014.

\bibitem{Elhay-Golub-Kautsky-32-143-92}
{\sc S.~Elhay, G.~H. Golub, and J.~Kautsky}, {\em {J}acobi matrices for sums of
  weight functions}, BIT, 32 (1992), pp.~143--166.

\bibitem{Elhay-Kautsky-6-205-94}
{\sc S.~Elhay and J.~Kautsky}, {\em {J}acobi matrices for measures modified by
  a rational factor}, Numer. Algor., 6 (1994), pp.~205--227.

\bibitem{Fasondini-Olver-Xu-21}
{\sc M.~Fasondini, S.~Olver, and Y.~Xu}, {\em Orthogonal polynomials on planar
  cubic curves}, Found. Comput. Math.,  (2021).

\bibitem{Fasondini-Olver-Xu-151-369-23}
{\sc M.~Fasondini, S.~Olver, and Y.~Xu}, {\em Orthogonal polynomials on a class
  of planar algebraic curves}, Stud. Appl. Math., 151 (2023), pp.~369--405.

\bibitem{Fiedler-Ptak-12-382-62}
{\sc M.~Fiedler and V.~Pt{\'a}k}, {\em On matrices with non-positive
  off-diagonal elements and positive principal minors}, Czechoslovak Math. J.,
  12 (1962), pp.~382--400.

\bibitem{Fiedler-Schneider-13-185-83}
{\sc M.~Fiedler and H.~Schneider}, {\em Analytic functions of ${M}$-matrices
  and generalizations}, Linear Algebra Appl., 13 (1983), pp.~185--201.

\bibitem{Frigo-Johnson-93-216-05}
{\sc M.~Frigo and S.~G. Johnson}, {\em The design and implementation of
  {FFTW3}}, Proc. IEEE, 93 (2005), pp.~216--231.

\bibitem{Gautschi-24-245-70}
{\sc W.~Gautschi}, {\em On the construction of {G}aussian quadrature rules from
  modified moments}, Math. Comput., 24 (1970), pp.~245--260.

\bibitem{Gautschi-36-547-81}
{\sc W.~Gautschi}, {\em Minimal solutions of three-term recurrence relations
  and orthogonal polynomials}, Math. Comp., 36 (1981), pp.~547--554.

\bibitem{Gautschi-04}
{\sc W.~Gautschi}, {\em Orthogonal Polynomials: Computation and Approximation},
  Clarendon Press, Oxford, UK, 2004.

\bibitem{Golub-Van-Loan-13}
{\sc G.~H. Golub and C.~F.~V. Loan}, {\em Matrix Computations}, The Johns
  Hopkins University Press, fourth~ed., 2013.

\bibitem{Goodman-et-al-35-233-95}
{\sc T.~N.~T. Goodman, C.~A. Micchelli, G.~Rodriguez, and S.~Seatzu}, {\em On
  the {C}holesky factorization of the {G}ram matrix of locally supported
  functions}, BIT Numer. Math., 35 (1995), pp.~233--257.

\bibitem{Goodman-et-al-18-331-98}
{\sc T.~N.~T. Goodman, C.~A. Micchelli, G.~Rodriguez, and S.~Seatzu}, {\em On
  the limiting profile arising from orthonormalizing shifts of exponentially
  decaying functions}, IMA J. Numer. Anal., 18 (1998), pp.~331--354.

\bibitem{Hansen-254-2092-08}
{\sc A.~C. Hansen}, {\em On the approximation of spectra of linear operators on
  {H}ilbert spaces}, J. Funct. Anal., 254 (2008), pp.~2092--2126.

\bibitem{Iserles-Webb-19-1191-19}
{\sc A.~Iserles and M.~Webb}, {\em Orthogonal systems with a skew-symmetric
  differentiation matrix}, Found. Comput. Math., 19 (2019), pp.~1191--1221.

\bibitem{Price-16-999-79}
{\sc T.~P. Jr.}, {\em Orthogonal polynomials for nonclassical weight
  functions}, SIAM J. Numer. Anal., 16 (1979), pp.~999--1006.

\bibitem{Kautsky-Golub-52-439-83}
{\sc J.~Kautsky and G.~H. Golub}, {\em On the calculation of {J}acobi
  matrices}, Linear Algebra Appl., 52 (1983), pp.~439--455.

\bibitem{Klippenstein-Slevinsky-403-113831-22}
{\sc B.~Klippenstein and R.~M. Slevinsky}, {\em Fast associated classical
  orthogonal polynomial transforms}, J. Comp. Appl. Math., 403 (2022),
  p.~113831.

\bibitem{Koornwinder-435-75}
{\sc T.~Koornwinder}, {\em Two-variable analogues of the classical orthogonal
  polynomials}, in Theory and Application of Special Functions, R.~Askey, ed.,
  1975, pp.~435--495.

\bibitem{Kowalski-13-309-82}
{\sc M.~A. Kowalski}, {\em The recursion formulas for orthogonal polynomials in
  $n$ variables}, SIAM J. Math. Anal., 13 (1982), pp.~309--315.

\bibitem{Krall-4-705-38}
{\sc H.~L. Krall}, {\em Certain differential equations for {T}chebycheff
  polynomials}, Duke Math. J., 4 (1938), pp.~705--718.

\bibitem{Kumar-28-769-74}
{\sc R.~Kumar}, {\em A class of quadrature formulas}, Math. Comp., 28 (1974),
  pp.~769--778.

\bibitem{Magnus-57-215-95}
{\sc A.~P. Magnus}, {\em Painlev\'e-type differential equations for the
  recurrence coefficients of semi-classical orthogonal polynomials}, J. Comp.
  Appl. Math., 57 (1995), pp.~215--237.

\bibitem{Mysovskikh-178-1252-68}
{\sc I.~P. Mysovskikh}, {\em On the construction of cubature formulas with the
  smallest number of nodes}, Dokl. Akad. Nauk SSSR, 178 (1968), pp.~1252--1254.

\bibitem{Nakatsukasa-Sete-Trefethen-40-A1494-18}
{\sc Y.~Nakatsukasa, O.~S\`ete, and L.~N. Trefethen}, {\em The {AAA} algorithm
  for rational approximation}, SIAM J. Sci. Comput., 40 (2018),
  pp.~A1494--A1522.

\bibitem{Newman-11-11-64}
{\sc D.~J. Newman}, {\em Rational approximation to $|x|$}, Michigan Math. J.,
  11 (1964), pp.~11--14.

\bibitem{Olver-et-al-NIST-10}
{\sc F.~W.~J. Olver, D.~W. Lozier, R.~F. Boisvert, and C.~W. Clark}, eds., {\em
  NIST Handbook of Mathematical Functions}, Cambridge U. P., Cambridge, UK,
  2010.

\bibitem{Olver-Nadakuditi-Trogdon-4-155002-15}
{\sc S.~Olver, R.~R. Nadakuditi, and T.~Trogdon}, {\em Sampling unitary
  ensembles}, Random Matrices: Theory and Applications, 4 (2015), p.~155002.

\bibitem{Olver-Slevinsky-Townsend-29-573-20}
{\sc S.~Olver, R.~M. Slevinsky, and A.~Townsend}, {\em Fast algorithms using
  orthogonal polynomials}, Acta Numerica, 29 (2020), pp.~573--699.

\bibitem{Olver-Townsend-55-462-13}
{\sc S.~Olver and A.~Townsend}, {\em A fast and well-conditioned spectral
  method}, SIAM Rev., 55 (2013), pp.~462--489.

\bibitem{Olver-Townsend-57-14}
{\sc S.~Olver and A.~Townsend}, {\em A practical framework for
  infinite-dimensional linear algebra}, in Proceedings of the First Workshop
  for High Performance Technical Computing in Dynamic Languages, 2014,
  pp.~57--62.

\bibitem{Olver-Xu-89-2847-20}
{\sc S.~Olver and Y.~Xu}, {\em Orthogonal polynomials in and on a quadratic
  surface of revolution}, Math. Comp., 89 (2020), pp.~2847--2865.

\bibitem{Olver-Xu-41-206-21}
{\sc S.~Olver and Y.~Xu}, {\em Orthogonal structure on a quadratic curve}, IMA
  J. Numer. Anal., 41 (2021), pp.~206--246.

\bibitem{Qi-56-105-84}
{\sc L.~Qi}, {\em Some simple estimates for singular values of a matrix},
  Linear Algebra Appl., 56 (1984), pp.~105--119.

\bibitem{Reed-Simon-80}
{\sc M.~Reed and B.~Simon}, {\em Methods of Modern Mathematical Physics},
  vol.~I: Functional Analysis, Academic Press, Inc., revised and enlarged~ed.,
  1980.

\bibitem{Skrzipek-41-331-92}
{\sc M.-R. Skrzipek}, {\em Orthogonal polynomials for modified weight
  functions}, J. Comp. Appl. Math., 41 (1992), pp.~331--346.

\bibitem{Slevinsky-GitHub-FastTransformsC}
{\sc R.~M. Slevinsky}, {\em {\tt
  https://github.com/MikaelSlevinsky/FastTransforms}}, GitHub,  (2018).

\bibitem{Slevinsky-47-585-19}
{\sc R.~M. Slevinsky}, {\em Fast and backward stable transforms between
  spherical harmonic expansions and bivariate {F}ourier series}, Appl. Comput.
  Harmon. Anal., 47 (2019), pp.~585--606.

\bibitem{Smith-19-33-65}
{\sc F.~J. Smith}, {\em An algorithm for summing orthogonal polynomial series
  and their derivatives with applications to curve-fitting and interpolation},
  Math. Comp., 19 (1965), pp.~33--36.

\bibitem{Llewellyn-Smith-Luca-475-20190105-19}
{\sc S.~G.~L. Smith and E.~Luca}, {\em Numerical solution of scattering
  problems using a {R}iemann--{H}ilbert formulation}, Proc. R. Soc. A, 475
  (2019), p.~20190105.

\bibitem{Snowball-Olver-145-3-20}
{\sc B.~Snowball and S.~Olver}, {\em Sparse spectral and $p$-finite element
  methods for partial differential equations on disk slices and trapeziums},
  Stud. Appl. Math., 145 (2020), pp.~3--35.

\bibitem{Snowball-Olver-5-1-21}
{\sc B.~Snowball and S.~Olver}, {\em Sparse spectral methods for partial
  differential equations on spherical caps}, Trans. Math. Appl., 5 (2021),
  pp.~1--37.

\bibitem{Szasz-Yeardley-8-621-58}
{\sc O.~Sz{\'a}sz and N.~Yeardley}, {\em The representation of an analytic
  function by general {L}aguerre series}, Pacific J. Math., 8 (1958),
  pp.~621--633.

\bibitem{Szego-75}
{\sc G.~Szeg{\H o}}, {\em Orthogonal Polynomials}, American Mathematical
  Society, Providence, Rhode Island, fourth~ed., 1975.

\bibitem{Trefethen-12}
{\sc L.~N. Trefethen}, {\em Approximation Theory and Approximation Practice},
  SIAM, Philadelphia, PA, 2012.

\bibitem{Trefethen-Bau-97}
{\sc L.~N. Trefethen and D.~B. III}, {\em Numerical Linear Algebra}, SIAM,
  1997.

\bibitem{Trogdon-Olver-36-174-16}
{\sc T.~Trogdon and S.~Olver}, {\em A {Riemann}--{Hilbert} approach to {Jacobi}
  operators and {Gaussian} quadrature}, IMA J. Numer. Anal., 36 (2016),
  pp.~174--196.

\bibitem{Uvarov-9-1253-69}
{\sc V.~B. Uvarov}, {\em The connection between systems of polynomials that are
  orthogonal with respect to different distribution functions}, Zh. Vychisl.
  Mat. Mat. Fiz., 9 (1969), pp.~1253--1262.

\bibitem{Vandebril-Van-Barel-Mastronardi-1-08}
{\sc R.~Vandebril, M.~V. Barel, and N.~Mastronardi}, {\em Matrix Computations
  and Semiseparable Matrices}, vol.~1: Linear Systems, Johns Hopkins University
  Press, Baltimore, MD, 2008.

\bibitem{Webb-Thesis-17}
{\sc M.~Webb}, {\em Isospectral algorithms, {T}oeplitz matrices and orthogonal
  polynomials}, PhD thesis, University of Cambridge, 2017.

\bibitem{Xu-342-855-94}
{\sc Y.~Xu}, {\em Block {J}acobi matrices and zeros of multivariate orthogonal
  polynomials}, Trans. Amer. Math. Soc., 342 (1994), pp.~855--866.

\bibitem{Xu-62-687-94}
{\sc Y.~Xu}, {\em Recurrence formulas for multivariate orthogonal polynomials},
  Math. Comp., 62 (1994), pp.~687--702.

\end{thebibliography}
